\documentclass[reqno]{amsart}
\usepackage{amsmath,mathtools,amssymb}

\mathtoolsset{showonlyrefs}

\usepackage{amsmath}
\usepackage{amsthm}
\usepackage{amsfonts}
\usepackage{amssymb}
\usepackage{mathrsfs}
\usepackage[shortlabels]{enumitem}
\usepackage{hyperref}

\newtheorem{theorem}{Theorem}
\newtheorem{proposition}[theorem]{Proposition}
\newtheorem{lemma}[theorem]{Lemma}
\newtheorem{corollary}[theorem]{Corollary}

\theoremstyle{definition}
\newtheorem{definition}{Definition}
\newtheorem{remark}{Remark}

\newcommand{\N}{\mathbb{N}}
\newcommand{\Z}{\mathbb{Z}}

\newcommand{\R}{\mathbb{R}}
\newcommand{\C}{\mathbb{C}}

\newcommand\e{\mathrm{e}}
\newcommand\I{\mathrm{i}}
\newcommand{\F}{\mathcal{F}}
\newcommand\re{\operatorname{Re}}
\newcommand\im{\operatorname{Im}}

\newcommand\spr{\operatorname{spr}}

\newcommand\eps\varepsilon
\renewcommand\epsilon\varepsilon
\renewcommand\rho\varrho
\newcommand\lm\lambda

\newcommand{\supp}{\operatorname{supp}}
\newcommand{\dist}{\operatorname{dist}}

\newcommand{\rd}{\mathrm{d}}

\newcommand{\beq}{\begin{equation}}
\newcommand{\eeq}{\end{equation}}
\newcommand{\be}{\begin{equation*}}
\newcommand{\ee}{\end{equation*}}
\newcommand{\bmat}{\begin{pmatrix}}
\newcommand{\emat}{\end{pmatrix}}

\begin{document}
\title[Random Schrödinger operators with complex decaying potentials]{Random Schrödinger operators with complex decaying potentials}
\subjclass[2020]{35P15, 81Q12, 35R60, 82B44.}
 \author[J.-C.\ Cuenin]{Jean-Claude Cuenin}
 \address[Jean-Claude Cuenin]{Department of Mathematical Sciences, Loughborough University, Loughborough,
 Leicestershire LE11 3TU, United Kingdom}
 \email{J.Cuenin@lboro.ac.uk}

 \author[K. Merz]{Konstantin Merz}
 \address[Konstantin Merz]{Institut f\"ur Analysis und Algebra, Technische Universit\"at Braunschweig, Universit\"atsplatz 2, 38106 Braunschweig, Germany}
 \email{k.merz@tu-bs.de}

\begin{abstract}
We prove that the eigenvalues of a continuum random Schrödinger operator $-\Delta+ V_{\omega}$ of Anderson type, with complex decaying potential, can be bounded (with high probability) in terms of an $L^q$ norm of the potential for all $q\leq d+1$. This shows that in the random setting, the exponent $q$ can be essentially doubled compared to the deterministic bounds of Frank (Bull. Lond. Math. Soc., 2011). This improvement is based on ideas of Bourgain (Discrete Contin. Dyn. Syst., 2002) related to almost sure scattering for lattice Schrödinger operators.
\end{abstract}

\maketitle

\section{Introduction and main result}

Consider a Schrödinger operator $-\Delta+ V$ on $L^2(\R^d)$. Frank \cite{MR2820160} proved the scale-invariant bounds
\begin{align}\label{LT bound}
|z|^{q-d/2}\lesssim \int_{\R^d}|V(x)|^{q}\rd x
\end{align} 
for eigenvalues $z$ of $-\Delta+V$, when $q\leq (d+1)/2$ (we call such $V$ short range). The short range condition is best possible, i.e.\ \eqref{LT bound} is generally not true for $q>(d+1)/2$. Counterexamples for $z>0$ were constructed by Frank and Simon \cite{MR3713021} and for $\im z\neq 0$ by Bögli and the first author \cite{bogli2021counterexample}. These counterexamples settle the Laptev--Safronov conjecture \cite{MR2540070} in the negative.

The aim of this paper is to show that for random potentials the short range exponent can be essentially doubled, from $(d+1)/2$ to $d+1$, compared to the deterministic case. We consider Anderson type Schrödinger operators of the form $-\Delta+V_{\omega}$, where
\begin{align}\label{Vomega}
V_{\omega}(x)=\sum_{j\in h\Z^d}\omega_jv_j\mathbf{1}_{Q}((x-j)/h),\quad Q=[0,1)^d,\quad h>0.
\end{align}
More generally, given a deterministic potential $V$, consider its randomization at scale $h>0$, given by
\begin{align}\label{randomization of deterministic V}
V_{\omega}(x)=\sum_{j\in h\Z^d}\omega_jV(x)\mathbf{1}_{Q}((x-j)/h).
\end{align}
One could also replace $\mathbf{1}_{Q_{j}}$ with some rapidly decaying function. Note that, in both cases \eqref{Vomega} and \eqref{randomization of deterministic V}, the $L^q$ norm of $V_{\omega}$ is deterministic, 
\begin{align}\label{Lq norm deterministic}
\|V_{\omega}\|_{L^q(\R^d)}=(h^d\sum_{j\in h\Z^d}|v_j|^q)^{1/q},
\end{align}
where $v_j$ is the $L^q$-average of $V$ over $j+hQ$ in the general case \eqref{randomization of deterministic V}, and we have $\|V_{\omega}\|_{L^q(\R^d)}=\|V\|_{L^q(\R^d)}$.
For this reason, we also denote the norm \eqref{Lq norm deterministic} by $\|V\|_{L^q(\R^d)}$ in case \eqref{Vomega}.
Crucially, we assume that $(\omega_j)_{j\in h\Z^d}\subset [-1,1]$ are \textbf{independent, mean-zero Gaussian or symmetric Bernoulli random variables}. In the following, $V_{\omega}$ will always denote the randomization \eqref{randomization of deterministic V} of a given deterministic potential $V$, and $\langle x\rangle=2+|x|$. The following standard assumptions on the local singularities of $V\in L^q(\R^d)$,
\begin{align}\label{lower bound q}
q\geq 1 \quad \mbox{if } d=1,\quad q>1\quad \mbox{if } d=2,\quad q\geq d/2 \quad \mbox{if } d\geq 3,
\end{align}
ensure that $-\Delta+V$ can be defined as an $m$-sectorial operator. These assumptions can be slightly weakened (see Remark \ref{rem. 1} (ii)) and only play a minor role here. In contrast, the average decay of the potential (i.e.\ an upper bound on $q$) -- to be stated in the assumptions of the following theorems -- is of central importance.

\begin{theorem}\label{thm. 1}
There exist constants $M_0,c>0$ such that the following holds.
For any $R,\lambda>0$, $0<h<R$, $|\epsilon|\ll \lambda$, $q\leq d+1$, for any $V\in L^q(\R^d)$ supported in a ball of radius $R$, and for any $M\geq M_0$, each eigenvalue $z=(\lambda+\I\epsilon)^2$ of $-\Delta+V_{\omega}$ satisfies 
\begin{align*}
\frac{\lambda^{2-\frac{d}{q}}}{\langle \lambda h\rangle^{d/2}(\log \langle \lambda R\rangle)^{7/2}}\leq M\|V\|_{L^{q}(\R^d)},
\end{align*}
except for $\omega$ in a set of measure at most
$\exp(-cM^2)$.
\end{theorem}

\begin{remark}\label{rem. 1}
(i) Outside the set $\lambda>0$, $|\epsilon|\ll \lambda$, obvious estimates (as in the case of real potentials) are available. These even hold for sums of powers of eigenvalues as in the classical Lieb--Thirring inequalities, see Frank--Laptev--Lieb--Seiringer \cite{MR2260376}.

\noindent (ii) As in \cite{MR4241156} (see also \cite{MR2219246}), one could weaken the local singularity assumption to $V\in L^{q_0}_{\rm loc}(\R^d)$, with $q_0$ satisfying \eqref{lower bound q}, and then replace $\|V\|_{L^q(\R^d)}$ by the right hand side of \eqref{Lq norm deterministic}, where $v_j$ is now the $L^{q_0}$-average of $V$ over $j+hQ$.
\end{remark}

\begin{remark}
There are three scales in the problem:
\begin{itemize}
\item The energy scale $\lambda^2$,
\item the scale $R$ measuring the support of the potential,
\item the randomization scale $h<R$.
\end{itemize}
In addition, we have introduced an arbitrary (dimensionless) parameter $M$ that appears in the large deviation bound.
\end{remark}

\begin{remark}\label{rem. counterexample 1}
Of course, a compactly supported potential of the form \eqref{Vomega} is in any $L^q$ space. 
The point of the estimate is the very weak dependence on $R$ (logarithmic), compared to what one would get by using Hölder's inequality and the deterministic bound \eqref{LT bound}.
Moreover, compactly supported potentials are interesting in view of the counterexample to the Laptev--Safronov conjecture of Bögli and the first author \cite{bogli2021counterexample}. The counterexample yields a sequence of potentials $V_\eps$, $\eps>0$ small, with $|V_\eps|\lesssim \eps \chi_{\epsilon}$, where $\chi_{\epsilon}$ is the indicator function of the tube 
$$T_{\epsilon}=\{(x_1,x'):\,|x_1|<\eps^{-1},|x'|<\eps^{-1/2}\},$$ such that $1+i\eps$ is an eigenvalue of $-\Delta+V_\eps$. Since 
\begin{align}\label{Veps}
\|V_{\epsilon}\|_{L^q(\R^d)}\lesssim \epsilon^{1-\frac{d+1}{2q}},
\end{align}
this shows that \eqref{LT bound} cannot hold for $q>(d+1)/2$. In this context, Theorem \ref{thm. 1} says that, after randomization on the scale
\begin{align*}
h\leq [\epsilon^{\frac{d+1}{2q}-1}\log(1/\epsilon)^{-\frac{7}{2}}]^{\frac{2}{d}},
\end{align*}
the counterexample for $(d+1)/2<q\leq d+1$ is almost surely destroyed. 
\end{remark}

\begin{remark}
Safronov \cite{Safronov2021} has recently considered eigenvalue sums for random Schrödinger operators with complex potentials of the same form as \eqref{Vomega}, but without the assumption on the distribution of $\omega_j$. However, these results do not give any new information about individual eigenvalues beyond what is known in the deterministic case \cite{MR2820160,MR3717979}. Moreover, Safronov's results only apply to the smaller range $q<(d+1)/2+1/(2d-4)$ compared to $q\leq d+1$. Our results are of a quite different character and therefore a direct comparison is not possible.
\end{remark}

The compact support assumption can be removed at the price of a tiny bit of pointwise decay. 

\begin{theorem}\label{thm. 2}
For any $\delta>0$ there exist constants $M_0,c>0$ such that the following holds.
For any $h,\lambda>0$, $|\epsilon|\ll \lambda$, $q\leq d+1$, for any $V\in \langle x\rangle^{-\delta}L^q(\R^d)$ and for any $M\geq M_0$, each eigenvalue $z=(\lambda+\I\epsilon)^2$ of $-\Delta+V_{\omega}$ satisfies 
\begin{align*}
\frac{\lambda^{2-\frac{d}{q}}}{\langle \lambda h\rangle^{d/2}(\log \langle \lambda h\rangle)^{2}}\leq M\|\langle \lambda x\rangle^{\delta} V\|_{L^{q}(\R^d)},
\end{align*}
except for $\omega$ in a set of measure at most
$\exp(-cM^2)$.
\end{theorem}

In fact, if we sacrifice the endpoint, we can also remove the pointwise decay assumption. 

\begin{theorem}\label{thm. 3}
For any $q<d+1$, there exist constants $M_0,c>0$ such that the following holds.
For any $h,\lambda>0$, $|\epsilon|\ll \lambda$, for any $V\in L^q(\R^d)$ and for any $M\geq M_0$, each eigenvalue $z=(\lambda+\I\epsilon)^2$ of $-\Delta+V_{\omega}$ satisfies 
\begin{align}\label{eq. thm. 3}
\frac{\lambda^{2-\frac{d}{q}}}{\langle \lambda h\rangle^{d/2}(\log \langle \lambda h\rangle)^{2}}\leq M\|V\|_{L^{q}(\R^d)},
\end{align}
except for $\omega$ in a set of measure at most
$\exp(-cM^2)$.
\end{theorem}

\begin{corollary}\label{cor. of thm. 3}
For $q<d+1$, we have
\begin{align*}
\sup_z\frac{|z|^{q-\frac{d}{2}}}{\|V\|_q^q}<\infty
\end{align*}
almost surely. The supremum is taken over all eigenvalues $z=(\lambda+\I\epsilon)^2$ of $-\Delta+V_{\omega}$ with $|\epsilon|\ll \lambda$.
\end{corollary}

\begin{proof}
Denote the supremum by $S$, and consider the events $E_M=\{S^{1/q}>M\}$. Since $E_{M}\supset E_{M+1}$ and $\mathbf{P}(E_{M_0})<\infty$, we have 
\begin{align*}
\mathbf{P}(S=\infty)=\lim_{M\to\infty}\mathbf{P}(E_M)=0.
\end{align*}
\end{proof}

\begin{remark}
The proof shows that Theorem \ref{thm. 3} (and hence Corollary \ref{cor. of thm. 3}) actually hold with $\|V\|_{L^q}$ replaced by the (smaller) Lorentz norm $\|V\|_{L^{q,\infty}}$. 
\end{remark}

The key technical elements in this work are estimates on certain ``elementary operators'', roughly of the form
\begin{align}\label{elem. op. intro}
R_0^{1/2}V_{\omega}R_0^{1/2},
\end{align}
where $R_0$ is the free resolvent at a fixed (complex) energy and $V_{\omega}$ is supported on a ball of radius $R>1$. In $d=2$, and in the discrete case (i.e.\ when $\Delta$ is replaced by the discrete Laplacian), Schlag--Shubin--Wolff \cite{MR1979773} proved\footnote{This is roughly the content of \cite[Lemma 3.9]{MR1979773}. Strictly speaking, the half powers of the resolvent are replaced by Fourier restriction and extension operators (or some mollified versions thereof), see also \cite[(1.12)]{MR1877824}.} that
the norm of these operators is bounded by a power of $\log R$. Their proof used in an essential way that the level sets corresponding to the symbol of $\Delta$ (the discrete Laplacian) are curved. Bourgain \cite{MR1877824} gave  a different proof using entropy bounds. His result is stated in $d=2$ but works in any dimension since it does not require curvature of the level sets (for the discrete Laplacian, these sets are not curved in higher dimensions). Motivated by work of Rodnianski--Schlag \cite{MR1941087}, he uses these bounds to prove almost sure existence and uniqueness of wave operators and a.c.\ spectrum (for energies away from the edges of the spectrum and zero). The result shaves off half a power of pointwise decay compared to the classical (deterministic) Agmon--Kato--Kuroda theory. In a follow-up work \cite{MR2083389}, Bourgain combined his method with the two-dimensional Stein--Tomas restriction theorem to obtain the same conclusion for potentials in $\langle x\rangle^{-\delta}\ell^3(\Z^2)$ ($\delta>0$ arbitrary). Note that there is a gap between the pointwise decay $\langle x\rangle^{-1/2}$ and $\ell^3(\Z^2)$. Bourgain \cite{MR2083389} observes that this gap cannot be overcome if one works with operators of the form \eqref{elem. op. intro} since the corresponding bounds (involving the $\ell^{3/2}(\Z^2)$ norm of the potential) are saturated (up to logarithms) by a Knapp example. Since the argument in \cite{MR2083389} is only sketched and is only stated in the two-dimensional discrete case, we will provide a complete proof of the optimality of our estimates (for the continuum multi-dimensional case) in Appendix \ref{Appendix B}\footnote{Bourgain's ideas and his Knapp example were also explained in a talk of Wilhelm Schlag at the Institute for Advanced Study on March 29, 2017.}
%
A representative (and simplified) example of these estimates, when $\lambda$ and $h$ are of unit size, is that
\begin{align}\label{elem. op. intro bound}
\|R_0^{1/2}V_{\omega}R_0^{1/2}\|\lesssim (\log R)^{\mathcal{O}(1)}\|V\|_{d+1}
\end{align}
with high probability (see Lemma \ref{lemma local resolvent bound} for a precise statement). Via a Born series argument (see Section \ref{section Born series} for details) this bound lead to a proof of Theorem \ref{thm. 1}. The proof of Theorem 2 then follows by a straightforward decomposition of the potential into dyadic shells $|x|\asymp 2^k$, similarly as in Bourgain's works \cite{MR1877824,MR2083389}. The proof of Theorem \ref{thm. 3} requires more effort and the argument presented in Section \ref{section Local to global arguments} is new to the best of our knowledge. The technique\footnote{Although Bourgain was almost certainly aware of these techniques, he did not bother to remove the logarithmic losses.} is reminiscent of an ``epsilon removal lemma" in the context of Fourier restriction theory (see e.g.\ \cite{MR1666558}). However, the technical implementation is a bit different since we are working with multilinear bounds (and with the resolvent instead of the Fourier restriction operator). 

While the bounds \eqref{elem. op. intro} are optimal (up to logarithms) in the sense that the Lebesgue exponent $d+1$ cannot be increased, it is an interesting open problem whether our eigenvalue estimates (say in the form of Corollary \ref{cor. of thm. 3}) are optimal. This problem is connected to a remark of Bourgain in \cite{MR2083389} that contains the idea of renormalizing away the self-energy interactions and then control the Born series via the sharp two-dimensional Fourier restriction theory of Carleson--Sjölin and Zygmund. This would amount to an $\ell^4(\Z^2)$ bound on the potential and would be natural and optimal from the point of view of restriction theory. A rigorous implementation of this idea seems difficult and has not been done so far, to the best of our knowledge.

\subsection*{Notation}
We write $A\lesssim B$ for two non-negative quantities $A,B\geq 0$ to indicate
that there is a constant $C>0$ such that $A\leq C B$. 
The dependence of the constant on fixed parameters like $d$ and $q$ is usually omitted (except in Section \ref{section Local to global arguments}).
The notation $A\asymp B$ means $A\lesssim B\lesssim A$.
The product measure associated to the $\omega_j$ is denoted by $\mathbf{P}$ and the expectation by $\mathbf{E}$. We denote the $L^p$ norm of a function $f$ in $\R^d$ by $\|f\|_{L^p(\R^d)}$. If the function is defined on a countable set $\Lambda$ we write $\|f\|_{\ell^p(\Lambda)}=(\sum_{\nu\in\Lambda}|f(\nu)|^p)^{1/p}$. If $\Lambda$ is finite, we also set $\|f\|_{\ell^p_{\rm av}(\Lambda)}=(|\Lambda|^{-1}\sum_{\nu\in\Lambda}|f(\nu)|^p)^{1/p}$. If it is clear from the context which norm is meant we sometimes use the abbreviation $\|f\|_p$. If $T:X\to Y$ is a bounded linear operator between two Banach spaces $X$ and $Y$,
we denote its operator norm by $\|T\|_{X\to Y}$. The indicator function of a set $\Omega\subset \R^d$ is denoted by $\mathbf{1}_{\Omega}$. For $1\leq p\leq\infty$ we denote its Hölder conjugate by $p'=(1-1/p)^{-1}$. An arbitrary ball of radius $R$ will be denoted by $B_R$, without specifying its center. We use the convention $\hat{f}(\xi)=\int_{\R^d}\e(-x\cdot\xi)f(x)\rd x$ for the Fourier transform of $f$, where $\e(x)=\e^{2\pi\I x}$, and $(f)^{\vee}(x)=\int_{\R^d}\e(x\cdot\xi)f(\xi)\rd \xi$ for the inverse Fourier transform.
Moreover, we recall the notation $\langle x\rangle=2+|x|$.

\subsection*{Organization}
In Section \ref{section Born series} we outline the rough top down strategy to prove our main results (see Proposition \ref{prop. spectral radius <1} for a summary). In Section \ref{Section Localization and discretization} we collect basic facts related to the uncertainty principle and recall the Stein-Tomas theorem for a discrete version of the Fourier extension operator that will play a major technical role in the proofs of the estimates in Section \ref{Section local bounds}. Section \ref{Section Randomization} is a short summary of probabilistic tools that will be used in the article. Section \ref{Section Entropy bound} fleshes out Bourgain's key idea of using entropy bounds. Section \ref{Section local bounds} contains the main local estimates and the completion of the proof of Theorem \ref{thm. 1}. Finally, in Section \ref{section Local to global arguments}, the local estimates are converted to global ones, leading to the proofs of Theorems \ref{thm. 2} and \ref{thm. 3}.

\section{Born series}\label{section Born series}
The proof of the eigenvalue estimates starts with the standard observation that $z\in\C\setminus[0,\infty)$ is an eigenvalue if and only if $I+R_0(z)V$ fails to be invertible as a bounded operator. This follows from the identity
\begin{align}\label{resolvent identity}
-\Delta+V-z=(-\Delta-z)(I+R_0(z)V).
\end{align}
Here we denoted the free resolvent operator $(-\Delta-z)^{-1}$ by $R_0(z)$ and we omitted the subscript $\omega$ on $V$. Similarly, we will denote the perturbed resolvent operator $(-\Delta+V-z)^{-1}$ by $R(z)$. 
To avoid confusion between the deterministic and the random potential we focus our attention on the Anderson type potentials \eqref{Vomega}.
In this case, the assumption that $V\in L^q(\R^d)$ already implies that $V$ is bounded (this follows from \eqref{Lq norm deterministic} and the fact that the $\ell^p$ spaces are nested). In particular, $R_0(z)V$ is a bounded operator. In the general case \eqref{randomization of deterministic V}, one truncates the potential at some fixed large level. Since the estimates of Theorems \ref{thm. 1}--\ref{thm. 3} are independent of the $L^{\infty}$ norm of $V$ and the truncated Schrödinger operator converges to the untruncated one in the norm resolvent sense, there is no loss of generality in assuming that the deterministic potential is bounded. In the following, we assume that $V$ is supported on a ball of radius $R$, i.e.\ the setting of Theorem \ref{thm. 1}. 
The case where $V$ is not compactly supported (Theorems \ref{thm. 2} and \ref{thm. 3}) will be considered in Section \ref{section Local to global arguments}.

Returning to \eqref{resolvent identity}, we see that $z$ cannot be an eigenvalue if the Born series 
\begin{align}\label{Born series}
R(z)=\sum_{n\in\N}(-1)^n[R_0(z)V]^nR_0(z)
\end{align}
converges, which is the case if the spectral radius of $R_0V$ is less than $1$. Consider the following multilinear expansion (omitting $z$),
\begin{align}\label{multilinear expansion Born series}
[R_0V]^n=\sum_{\sigma_1,\ldots,\sigma_n}R_0^{\sigma_1}VR_0^{\sigma_2}V\ldots R_0^{\sigma_n}V
\end{align}
where $\sigma_j\in\{\rm{low},\rm{high}\}$. Here, $R_0^{\rm low}$ is the resolvent (smoothly) localized to frequencies in $B(0,2)$ and $R_0^{\rm high}=R_0-R_0^{\rm low}$.
Since we are dealing with scale-invariant estimates, we may assume without loss of generality that $\lambda=1$, hence $z=(1+\I \epsilon)^2$.
Then each summand is a composition of operators of the form
$C^{(\delta_2)}VC^{(\delta_1)}$, where $C^{(\delta)}$ denotes a function satisfying a bound
\begin{align}\label{Cdelta bound}
|C^{(\delta)}(\xi)|\leq (||2\pi \xi|^2-1|+\delta)^{-1/2}
\end{align}
and the corresponding Fourier multiplier is denoted by the same symbol.
Clearly, the bound
\eqref{Cdelta bound} holds with $\delta=1$ for $C^{(\delta)}=(R_0^{\rm high})^{1/2}$ or $C^{(\delta)}=|R_0^{\rm high}|^{1/2}$. 
In Section~\ref{subsection smoothing} we will show that \eqref{Cdelta bound} holds with $\delta=1/R$ if $C^{(\delta)}$ is a mollification of $(R_0^{\rm low})^{1/2}$ or $|R_0^{\rm low}|^{1/2}$ at scale $1/R$. Such a mollification can always be performed (except for the first resolvent in the Born series, but this does not affect convergence), due to the localizing effect of the potential, which we assumed to be supported in a ball of radius $R$. 
The spectral radius is given by  Gelfand's formula, $\spr(R_0V)=\lim_{n\to\infty}\|[R_0V]^n\|^{1/n}$. Thus, in view of the previous discussion, we have
$\spr(R_0V)\leq \sup \|C^{(\delta_2)}VC^{(\delta_1)}\|$, where the supremum is taken over all functions satisfying~\eqref{Cdelta bound}. We will ignore the high frequency part of the resolvent $R_0^{\rm high}$
from now on since there are obvious elliptic estimates available for this part. We may thus restrict our attention to functions as in \eqref{Cdelta bound} that are compactly supported in $B(0,2)$. We summarize the observations of this Section in the following proposition.

\begin{proposition}\label{prop. spectral radius <1}
Let $z=(1+\I \epsilon)^2$, with $|\epsilon|\ll 1$. Let $V$ be supported in a ball of radius $R$. If (for a given realization of $\omega$)
\begin{align}\label{elementary operators}
\|C^{(\delta_2)}VC^{(\delta_1)}\|\leq c<1
\end{align}
for all functions $C^{(\delta_i)}$, $i=1,2$, satisfying~\eqref{Cdelta bound} with $\delta_1,\delta_2=1/R$ and supported in $B(0,2)$, then $z$ is not an eigenvalue of $-\Delta+V$. 
\end{proposition}

We refer to operators of the form \eqref{elementary operators} as ``elementary operators" since they form the building blocks of the Born series. We prove norm estimates on these and related operators in Section \ref{Section local bounds}. These estimates are the key technical elements in this work.

\begin{remark}
Strictly speaking, the previous argument is only valid for $\epsilon\neq 0$, but there are techniques to extend this to embedded eigenvalues ($\epsilon=0$), see e.g.\ \cite[Prop. 3.1]{MR3713021}.
\end{remark}

\begin{remark}
Later on, we will assume that all functions $C^{(\delta)}$ are supported in a small neighborhood of the unit sphere. This does not affect the validity of the above argument.
\end{remark}

\section{Localization and discretization}\label{Section Localization and discretization}

\subsection{Localization in momentum space}
Denote by $\mathcal{Q}_h$ the collection of all cubes $Q_h$ of sidelength $h$.
Define the weight function
\begin{align}\label{def. wQ}
w_{Q_h}(x)=(1+h^{-1}\dist(x,Q_h))^{-100d},\quad x\in \R^d,\quad Q_h\in\mathcal{Q}_h.
\end{align}

\begin{lemma}\label{lemma locally constant}
Let $v\in\mathcal{S}(\R^d)$ and assume that $\hat{v}$ is supported in $B(0,1/h)$. Then $v$ is locally constant on cubes $Q_h$ of sidelength $h$ in the sense that 
\begin{align*}
\|v\|_{L^{\infty}(Q_h)}\lesssim |Q_h|^{-1}\|v\|_{L^1(w_{Q_h})}.
\end{align*}
\end{lemma} 

\begin{proof}
By scaling, it suffices to prove this for $h=1$.
Choose $\eta\in \mathcal{S}(\R^d)$ such that $\eta=1$ on $B(0,1)$. Then we have $\hat{v}=\eta\hat{v}$ and hence $v=(\eta)^{\vee}\ast v$. Since $(\eta)^{\vee}\in \mathcal{S}(\R^d)$, it follows that
\begin{align*}
\kappa_w=\sup_{Q\in \mathcal{Q}_1}\sup_{(x,y)\in Q\times\R^d}|(\eta)^{\vee}(x-y)|w_Q(y)^{-1}<\infty,
\end{align*}
where the first supremum is taken over all cubes of sidelength one. Thus, for any cube $Q$ of sidelength one and for $x\in Q$, we have 
\begin{align*}
|v(x)|\leq \int_{\R^d} |(\eta)^{\vee}(x-y)||v(y)|\rd y\leq \kappa_w\|v\|_{L^1(w_Q)}.
\end{align*}
Taking the supremum over $x\in Q$ proves the claim.
\end{proof}

\begin{lemma}\label{lemma Ph}
Let $v\in\mathcal{S}(\R^d)$ and assume that $\hat{v}$ is supported in $B(0,1/h)$.
Let $\Lambda_h\subset \R^d$ be a set of $h$-separated points. Then for any $p\geq 1$, we have
\begin{align*}
\|v\|_{\ell^p(\Lambda_h)}\lesssim h^{-d/p}\|v\|_{L^p(\R^d)}.
\end{align*}
\end{lemma}

\begin{proof}
Again by scaling, we can assume $h=1$. Thus, let 
$\Lambda\subset \R^d$ be a set of $1$-separated points. Pick a collection of cubes $Q$ of sidelength one that cover $\Lambda$. By Lemma \ref{lemma locally constant}, 
\begin{align*}
\|v\|_{\ell^p(\Lambda)}^p=\sum_{\nu\in \Lambda}|v(\nu)|^p\lesssim 
\sum_{Q}\|v\|_{L^1(w_{Q})}^p,
\end{align*}
Write $v=\sum_{Q'}v_{Q'}$, where $v_{Q'}$ is supported on $Q'$. Then 
\begin{align*}
\|v_{Q'}\|_{L^1(w_{Q})}\leq (1+\dist(Q,Q'))^{-100d}\|v_{Q'}\|_{L^1(\R^d)}.
\end{align*}
By Hölder, $\|v_{Q'}\|_{L^1(\R^d)}\leq \|v_{Q'}\|_{L^p(\R^d)}$. Hence,
\begin{align*}
\sum_{Q}\|v\|_{L^1(w_{Q})}^p\lesssim \sum_{Q,Q'}(1+\dist(Q,Q'))^{-100dp}\|v_{Q'}\|_{L^p(\R^d)}^p\lesssim \|v\|_{L^p(\R^d)}^p,
\end{align*}
where we summed a geometric series in $Q$. 
\end{proof}

\subsection{Localization in position space}

We will make use of the following standard device in local restriction theory (see e.g.\ \cite[Lemma 1.26]{MR3971577}).

\begin{lemma}\label{lemma bump fct. phi}
There exists a bump function $\phi$ on $\R^d$ with $\supp\phi\subset B(0,1)$ and with non-negative Fourier transform satisfying $\mathbf{1}_{B(0,1)}\leq \hat{\phi}$. Moreover, $\hat{\phi}$ is an even function. 
\end{lemma}

It is clear that the rescaled function $\phi_R(\xi)=R^d\phi(R\xi)$ satisfies
\begin{align*}
\supp \phi_R\subset B(0,R^{-1}),\quad \mathbf{1}_{B(0,R)}\leq \hat{\phi}_R.
\end{align*}

Let $M_{\lambda}=\{\xi\in \R^d:\,|\xi|=\lambda\}$, and consider the extension operator
\begin{align*}
\mathcal{E}_{\lambda}:L^2(M_{\lambda},\rd\sigma_{\lambda})\to L^{\infty}(\R^d),\quad (\mathcal{E}_{\lambda}g)(x)=(g\rd\sigma_{\lambda})^{\vee}(x),
\end{align*}
where $\sigma_{\lambda}$ is surface measure on $M_{\lambda}$. We write $\mathcal{E}\equiv\mathcal{E}_1$ and $M\equiv M_1$, $\sigma\equiv \sigma_1$. 

\subsection{Discrete Fourier extension operator}

\begin{definition}
Let ${\rm Discres}(M,p,2)$ be the best constant such that the following hold for each $R\geq 2$, each collection $\Lambda^*_R$ consisting of $1/R$-separated points on $M$, each sequence $a_{\nu}\subset \C$, each ball $B_R$ and each collection $\Lambda_1$ of $1$-separated points in $\R^d$:
\begin{align}\label{def. discres}
\|\sum_{\nu\in \Lambda^*_R}a_{\nu}\e(\nu\cdot x)\|_{\ell^{p'}(\Lambda_1\cap B_R)}\leq {\rm Discres}(M,p,2)R^{\frac{d-1}{2}}\|a_{\nu}\|_{\ell^2(\Lambda^*_R)}.
\end{align}
\end{definition}

\begin{proposition}\label{prop. continuous-discrete extension}
If $1\leq p\leq \infty$, then
\begin{align}\label{eq. continuous-discrete extension}
{\rm Discres}(M,p,2)\lesssim \|\mathcal{E}\|_{L^2(M,\rd\sigma)\to L^{p'}(\R^d)}.
\end{align}
Moreover, if $p\geq 2$, then the reverse inequality also holds.
\end{proposition}

\begin{proof}
The claim is a special case of \cite[Prop. 1.29]{MR3971577}, with one small difference. There, ${\rm Discres}(M,p,2)$ is defined with the $L^{p'}(B_R)$ norm in the left hand side of~\eqref{def. discres}. Thus, let ${\rm Discres}'(M,p,2)$ be the best constant in the inequality
\begin{align}\label{Demeter discres}
\|\sum_{\nu\in \Lambda^*_R}a_{\nu}\e(\nu\cdot x)\|_{L^{p'}(B_R)}\leq {\rm Discres}'(M,p,2)R^{\frac{d-1}{2}}\|a_{\nu}\|_{\ell^2(\Lambda^*_R)}.
\end{align}
Then \cite[Prop. 1.29]{MR3971577} asserts that the proposition holds with ${\rm Discres}'(M,p,2)$ in place of ${\rm Discres}(M,p,2)$.
Thus, \eqref{eq. continuous-discrete extension} follows once we show that
\begin{align}\label{claim discres'<discres}
{\rm Discres}'(M,p,2)\gtrsim {\rm Discres}(M,p,2).
\end{align}
Without loss of generality we may assume that $B_R=B(0,R)$. If we set
\begin{align*}
f(x)=\sum_{\nu\in \Lambda^*_R}a_{\nu}\e(\nu\cdot x),\quad \mbox{then}\quad
\F(f\hat{\phi}_R)(\xi)=\sum_{\nu\in \Lambda^*_R}a_{\nu}\phi_R(\xi+\nu),
\end{align*}
where $\phi_R$ is as before and $\mathcal{F}$ denotes the Fourier transform. Note that $\F(f\hat{\phi}_R)=\hat{f}\ast \phi_R$ is supported in an $1/R$-neighborhood of $M$. 
In particular, it is supported on the ball $B(0,2)$. Thus, for any collection $\Lambda_1$ of $1$-separated points in $\R^d$,
\begin{align*}
\|f\|_{\ell^{p'}(\Lambda_1\cap B_R)}\leq \|f\hat{\phi}_R\|_{\ell^{p'}(\Lambda_1)}\lesssim \|f \hat{\phi}_R\|_{L^{p'}(\R^d)},
\end{align*}
where we used $\hat{\phi}_R\geq \mathbf{1}_{B_R}$ in the first inequality and Lemma \ref{lemma Ph} in the second. 
By a partition of unity and a sparsification argument we may assume that $f$ is supported on a disjoint union of balls of radius $R$. By the rapid decay of $\hat{\phi}_R$ and by the definition of ${\rm Discres}'(M,p,2)$,
\begin{align*}
\|\hat{\phi}_Rf\|_{L^{p'}(\R^d)}\lesssim_N \sum_{j=1}^{\infty}j^{-N}\|f\|_{L^{p'}(B(x_j,R))}\lesssim {\rm Discres}'(M,p,2)R^{\frac{d-1}{2}}\|a_{\nu}\|_{\ell^2(\Lambda^*_R)},
\end{align*}
where we used that \eqref{Demeter discres} holds uniformly in the centers of the balls. Combining the last two estimates yields \eqref{claim discres'<discres}.

To prove the reverse inequality to \eqref{eq. continuous-discrete extension}, we may assume that $B_R=B(0,R)$. By \cite[Prop. 1.29]{MR3971577} it suffices to prove the reverse inequality to \eqref{claim discres'<discres}. Let $\Lambda_1$ be a $1$-net of points $x_j\in B_R$. Let $f(x)$ be defined as above. Without loss of generality we may assume that $f$ is supported on a disjoint collection of balls $B(x_j,10)$. Then
\begin{align*}
\|f\|_{L^{p'}(B_R)}&=(\sum_j\|f\|^{p'}_{L^{p'}(B(x_j,10))})^{1/p'}=(\int_{B(0,10)}\sum_j|f(x_j+y)|^{p'}\rd y)^{1/p'}
\\
&\lesssim {\rm Discres}(M,p,2)R^{\frac{d-1}{2}}\|a_{\nu}\|_{\ell^2(\Lambda^*_R)},
\end{align*}
where we used that \eqref{def. discres} holds for each collection $x_j+y$ of $1$-separated points, uniformly in $y$.
\end{proof}

\subsection{Stein-Tomas theorem}

The following is an immediate consequence of the 
Stein-Tomas theorem and Proposition \ref{prop. continuous-discrete extension} (see also \cite[Cor. 1.30]{MR3971577}).

\begin{proposition}\label{prop. Stein-Tomas Discres}
Let $p'\geq 2(d+1)/(d-1)$. Then ${\rm Discres}(M,p,2)\lesssim 1$.
\end{proposition}

\section{Randomization}\label{Section Randomization}

\subsection{Sub-gaussian random variables}
We recall that a (complex) scalar random variable $X$ is called \emph{sub-gaussian} if it has finite sub-gaussian norm,
\begin{align*}
\|X\|_{\psi_2}=\inf\{t>0:\,\mathbf{E}\exp(|X|^2/t^2)\leq 2\}<\infty.
\end{align*}
We will need the following elementary properties of sub-gaussian (e.g.\ Gaussian or symmetric Bernoulli) random variables. (see e.g.\ \cite[Proposition 2.6.1 and Exercise 2.5.10]{MR3837109}).

\begin{proposition}\label{prop. properties of subgaussian rv}
Assume that $(X_j)_{j=1}^N$, $N\geq 2$, is a finite collection of i.i.d.\ mean-zero sub-gaussian random variables. 
\begin{itemize}
\item[i)] Then $\sum_{j=1}^N X_j$ is also sub-gaussian, and
\begin{align*}
\|\sum_{j=1}^N X_j\|_{\psi_2}^2\lesssim \sum_{j=1}^N\|X_j\|_{\psi_2}^2.
\end{align*} 
\item[ii)] We have
\begin{align*}
\mathbf{E}\max_{j\leq N}|X_j|\lesssim \sqrt{\log N}\max_{j\leq N}\|X_j\|_{\psi_2}.
\end{align*}
\end{itemize}
\end{proposition}

\begin{proof}
The claim follows by applying \cite[Prop. 2.6.1 and Ex. 2.5.10]{MR3837109} to $\re X_j$ and $\im X_j$ separately.
\end{proof}

\subsection{Tail bounds}

We now consider tail bounds for
\emph{vector-valued} Gaussian or Bernoulli random variables $X$. We have $(\mathbf{E}\|X\|^p)^{1/p}\asymp(\mathbf{E}\|X\|^q)^{1/q}$ for all $p,q>0$
  (cf.~\cite[Corollary 3.2 and Theorem 4.7]{MR1102015}),
  which, combined with \cite[(3.5), (4.12)]{MR1102015} implies
  \begin{align*}
    \mathbf{P}(\|X\| > t)
    \leq \exp\left(-\frac{ct^2}{(\mathbf{E}\|X\|)^2}\right)
  \end{align*}
for some $c>0$. Thus the following lemma is obvious.

\begin{lemma}\label{lemma tail bound}
If $\mathbf{E}\|X\|\leq C$, then
  \begin{align*}
    \mathbf{P}(\|X\| > MC)
    \leq \exp(-cM^2)
  \end{align*}
 for any $M>0$. 
\end{lemma}

\section{Entropy bound}\label{Section Entropy bound}

Consider a linear operator $S:\mathcal{H}\to \ell^{\infty}_m$, where $\mathcal{H}$ is a finite-dimensional Hilbert space and $\ell^{\infty}_m=\ell^{\infty}(\{1,\ldots,m\})$. For $\epsilon>0$ let $\mathcal{N}(\epsilon)$ be the minimal number of balls in $\ell^{\infty}_m$ of radius $\epsilon$ needed to cover the set $\{Sx:\,x\in \mathcal{H},\,\|x\|_{\mathcal{H}}\leq 1\}$. Here we use the convention that the centers of the balls are contained in the set they cover (i.e.\ $\mathcal{N}(t)$ is the \emph{covering number} as opposed to the \emph{exterior covering number}, see e.g.\ \cite[Sect. 4.2]{MR3837109}). Using an entropy bound known as the ``dual Sudakov inequality", which is attributed to Pajor and Tomczak-Jaegermann \cite{MR845980}, Bourgain \cite[(4.2)]{MR1877824} shows that
\begin{align}\label{dual Sudakov}
\log\mathcal{N}(\epsilon)\lesssim (\log m)\epsilon^{-2}\|S\|_{\mathcal{H}\to \ell^{\infty}_m}^2.
\end{align}
The quantity $\log\mathcal{N}(\epsilon)$ is called the \emph{entropy number} of the image of the unit ball in $\mathcal{H}$ under the map $S$. The crucial observation is that \eqref{dual Sudakov} is independent of $\dim\mathcal{H}$.
We apply this bound to the operator featuring in \eqref{def. discres}, i.e.
\begin{align}\label{def. S}
S:\ell^2_{\rm av}(\Lambda^*_R)\to \ell^{\infty}(\Lambda_1\cap B_R),\quad \{a_\nu\}\mapsto
\{\sum_{\nu\in \Lambda^*_R}a_{\nu}\e(\nu\cdot x)\}_x
\end{align} 
In this case, we have $\mathcal{H}=\ell^2_{\rm av}(\Lambda^*_R)$ and $\ell^{\infty}_m=\ell^{\infty}(\Lambda_1\cap B_R)$. In particular, we have $m\asymp R^d$. Here and in the following we always assume $R\geq 2$. Proposition \ref{prop. Stein-Tomas Discres} gives
\begin{align}\label{boundedness of S Stein Tomas}
\|S\|_{\ell^2_{\rm av}(\Lambda^*_R)\to \ell^{p'}(\Lambda_1\cap B_R)}\lesssim 1\quad \mbox{for}\quad p'\geq 2(d+1)/(d-1).
\end{align}
In particular, we have the trivial bound ($p'=\infty$)
\begin{align}\label{boundedness of S trivial}
\|S\|_{\ell^2_{\rm av}(\Lambda^*_R)\to \ell^{\infty}(\Lambda_1\cap B_R)}\lesssim 1.
\end{align}
Combining the latter with \eqref{dual Sudakov} yields the following entropy bound.

\begin{proposition}\label{prop. entropy bound}
Let $S$ be given by \eqref{def. S}. The entropy number satisfies the bound
\begin{align*}
\log\mathcal{N}(\epsilon)\lesssim (\log R)\epsilon^{-2}.
\end{align*}
\end{proposition}

\begin{corollary}\label{cor. representaion}
Let $p'\geq 2(d+1)/(d-1)$. For every $k\in \Z_+$, 
there exist sets $\mathcal{F}_k\subset \ell^{\infty}(\Lambda_1\cap B_R)$ with the following properties.
\begin{enumerate}
\item[\rm(a)] $\log|\mathcal{F}_k|\lesssim  \log(R)4^k$ (here $|\cdot|$ denotes counting measure).
\item[\rm(b)] For $\xi\in \mathcal{F}_k$,
\begin{align}
\|\xi\|_{\ell^{\infty}(\Lambda_1)}\lesssim 2^{-k},\quad \|\xi\|_{\ell^{p'}(\Lambda_1)}\lesssim 1.
\end{align} 
\item[\rm(c)] For each $a\in \ell^2_{\rm av}(\Lambda^*_R)$ with $\|a\|_{\ell^2_{\rm av}(\Lambda^*_R)}\leq 1$ there is a representation
\begin{align*}
Sa=\sum_{k\in \Z_+}\xi^{(k)}\quad \mbox{for some}\quad \xi^{(k)}\in\mathcal{F}_k.
\end{align*} 
\end{enumerate}
\end{corollary}

\begin{proof}
We follow Bourgain \cite[page 75-76]{MR2083389}, but provide more details (note also that there is a misprint in (3.13) there; it should be $4^r$, not $4^{-r}$). This is a standard chaining argument.

We start by noting that, in view of \eqref{dual Sudakov} and \eqref{boundedness of S trivial}, we have $\mathcal{N}(C)=1$ for $C$ sufficiently large. In the following (and only in this proof) denote the unit ball in $\ell^2_{\rm av}(\Lambda^*_R)$ by $B_1$. Similarly, $B(\xi,\epsilon)$ denotes a ball centered at $\xi$ and with radius $\epsilon$ in $\ell^{\infty}(\Lambda_1\cap B_R)$. We also write $\|\cdot\|_{p}=\|\cdot\|_{\ell^p(\Lambda_1)}$ here. By possibly rescaling $SB_1$ by a constant, we may assume that $C=1$. Thus, we have $\mathcal{N}(1)=1$. We get, by Proposition \ref{prop. entropy bound},
\begin{align}\label{Ntk}
\log\mathcal{N}(2^{-k})\lesssim \log(R)4^k.
\end{align}
Thus, for each $k\geq 0$, there exist subsets $\mathcal{E}_k\subset\ell^{\infty}(\Lambda_1\cap B_R)$ of cardinality $\mathcal{N}(2^{-k})$ satisfying  
\begin{align*}
SB_1\subset \bigcup_{\xi\in\mathcal{E}_k}B(\xi,2^{-k}).
\end{align*}
Applying these nets for each $k$, we can assign to each element $Sa\in SB_1$ a chain $\{\xi_k\}$
 converging to $Sa$, with $\xi_k\in \mathcal{E}_k$ and 
\begin{align}\label{def. Fk}
 \|\xi_k-\xi_{k-1}\|_{\infty}\leq 2^{-k}+2^{1-k}
\end{align} 
for all $k$.
By telescoping, we have 
\begin{align*}
Sa=\xi_0+\lim_{N\to\infty}\sum_{k=1}^N(\xi_k-\xi_{k-1})
\end{align*}
Thus, we may choose $\mathcal{F}_{0}=\mathcal{E}_{0}$ and $\mathcal{F}_k\subset \mathcal{E}_k-\mathcal{E}_{k-1}$, $k>0$, as the collection of all vectors $\xi^{(k)}=\xi_k-\xi_{k-1}$ for which \eqref{def. Fk} holds. Since the difference set $\mathcal{E}_k-\mathcal{E}_{k-1}$ has cardinality $|\mathcal{E}_k||\mathcal{E}_{k-1}|$ the claimed properties hold by construction.
\end{proof}

\section{Local bounds on elementary operators}\label{Section local bounds}

\subsection{Local extension bound}
Let $h,R>0$. Consider $V_{\omega}$ of the form \eqref{randomization of deterministic V}, where $V$ is a given deterministic potential supported in $B_R$.
Also fix $p'\geq 2(d+1)/(d-1)$ and define $q$ by $1/q=1/p-1/p'$. Note that this convention differs from that in the main theorems by a change of variables $q\to 2q$.

\begin{lemma}\label{lemma local extension bound}
Under the above assumptions, 
we have
\begin{align*}
\mathbf{E}\|\mathcal{E}^*V_{\omega} \mathcal{E}\|_{L^2(M,\rd\sigma)\to L^2(M,\rd\sigma)}\lesssim \langle h\rangle^{d/2}(\log \langle  R\rangle)^{1/2}(\log \langle h\rangle + \log \langle R\rangle)^{2}\|V\|_{L^{2q}(\R^d)}.
\end{align*}
\end{lemma}

\begin{proof}
Since the right hand side only gets larger if we replace $R$ and $h$ by $R+2$ and $h+2$, respectively, we may assume $R,h\geq 2$.
We first observe that 
\begin{align}\label{smoothing identity for V}
\mathcal{E}^*V_{\omega} \mathcal{E}=\mathcal{E}^*(V_{\omega}\ast\varphi) \mathcal{E}
\end{align}
for any Schwartz function $\varphi$ satisfying $\hat{\varphi}=1$ on $B(0,2)$. We can thus assume without loss of generality that $V$ is smooth on the unit scale.
Let $g,g'$ be unit vectors in $L^2(M,\rd\sigma)$. Then
\begin{align*}
\langle \mathcal{E}^*V_{\omega} \mathcal{E}g,g'\rangle&=\sum_{j\in h\Z^d}\omega_j\int_{Q_h+j}\overline{V(x)(\mathcal{E}g)(x)}(\mathcal{E}g')(x)\rd x,
\end{align*}
where $Q_h=[0,h)^d$. Let $\Lambda^*_R=\{\eta_{\nu}\}$ be a $1/R$-net in $M$. By working with a partition of unity, we may assume that $g$ is supported on a collection of disjoint balls $B(\eta_{\nu},10/R)$. After a change of variables $g(\eta)=g(\eta_{\nu}+\tau)$, we may write
\begin{align*}
\mathcal{E}g(x)=\sum_{\nu}\int_{M\cap B(0,10/R)} \e(x\cdot(\eta_{\nu}+\tau
))g(\eta_{\nu}+\tau)\rd\tau,
\end{align*} 
where $\rd\tau$ denotes surface measure, and similarly (summing over a possibly different index set)
\begin{align*}
\mathcal{E}g'(x)=\sum_{\nu'}\int_{M\cap B(0,10/R)} \e(x\cdot(\eta_{\nu'}+\tau'
))g'(\eta_{\nu'}+\tau')\rd \tau'.
\end{align*} 
Similarly to the change of variables $\eta=\eta_{\nu}+\tau$ in the domain, we change variables $x=x_i+y$ in the target. Here, $\Lambda_1=\{x_i\}$ is a $1$-net in $\R^d$. Hence, for any integrable function $F:\R^d\to\C$, supported on a disjoint collection of balls $B(x_i,10)$, 
\begin{align*}
\int_{\R^d}F(x)\rd x=\sum_i\int_{B(0,10)}F(x_i+y)\rd y.
\end{align*}
Using a partition of unity we may sparsify the potential, so that the above holds for 
\begin{align*}
F_j(x)=\overline{V(x)(\mathcal{E}g)(x)}(\mathcal{E}g')(x)\mathbf{1}_{Q_h+j}(x).
\end{align*}
Note that in this case the sum is restricted to those $i$ satisfying $x_i\in B(j,10+h)$. For fixed $\tau\in B(0,10/R)$ and $y\in B(0,10)$ we consider the discrete extension operator 
\begin{align*}
S:\ell^2_{\rm av}(\Lambda^*_R)\to \ell^{\infty}(\Lambda_1\cap B_R),\quad \!\!\!\!\!\!\{g(\eta_{\nu}+\tau)\}_{\nu}\mapsto
\{\sum_{\nu}\e((x_i+y)\cdot (\eta_{\nu}+\tau))g(\eta_{\nu}+\tau)\}_i.
\end{align*} 
Note that the points $\mu_{\nu}=\eta_{\nu}+\tau$ and $z_i=x_i+y$ form a $1/R$-separated set in $M$ and a $1$-separated set in $\R^d$, respectively, so that \eqref{boundedness of S Stein Tomas}, \eqref{boundedness of S trivial} hold.
Using Corollary~\ref{cor. representaion}, we can find a representation (note that the vectors $\xi^{(k)}$ depend on $\tau,y$)
\begin{align*}
\sum_{\nu}\e((x_i+y)\cdot (\eta_{\nu}+\tau))g(\eta_{\nu}+\tau)=\sum_{k\in \Z_+}\xi_i^{(k)},\quad \xi^{(k)}\in\mathcal{F}_k,
\end{align*}
with bounds
\begin{align}\label{bounds xi in terms of g}
\|\xi^{(k)}\|_{\infty}\lesssim 2^{-k}\|g(\eta_{\nu}+\tau)\|_{\ell^2_{\nu,{\rm av}}},\quad \|\xi^{(k)}\|_{p'}\lesssim \|g(\eta_{\nu}+\tau)\|_{\ell^2_{\nu,{\rm av}}}
\end{align}
for all $k\in\Z_+$ and $y\in B(0,10)$. Similarly, there is a representation
\begin{align*}
\sum_{\nu'}\e((x_i+y)\cdot (\eta_{\nu'}+\tau'))g'(\eta_{\nu'}+\tau')=\sum_{k'\in \Z_+}\xi_i^{(k')},\quad \xi^{(k')}\in\mathcal{F}_{k'},
\end{align*}
with bounds
\begin{align}\label{bounds xi' in terms of g'}
\|\xi^{(k')}\|_{\infty}\lesssim 2^{-k'}\|g'(\eta_{\nu'}+\tau')\|_{\ell^2_{\nu',{\rm av}}},\quad \|\xi^{(k')}\|_{p'}\lesssim \|g'(\eta_{\nu'}+\tau')\|_{\ell^2_{\nu',{\rm av}}}.
\end{align}
The above observations lead to the estimate
\begin{align*}
|\langle \mathcal{E}^*V_{\omega} \mathcal{E}g,g'\rangle|\leq 
\sum_{k,k'\in \Z_+}\int \max_{(\xi,\xi')\in\mathcal{F}_k\times \mathcal{F}_{k'}}|\sum_{j\in h\Z^d}\sum_i \omega_j\overline{V(x_i+y)\xi_i}\xi_i'|\rd y\rd\tau\rd\tau',
\end{align*}
where the integral is taken over $(y,\tau,\tau')\in B(0,10)\times (M\cap B(0,10/R))^2$ and the sum over $i$ is restricted to $x_i+y\in Q_h+j$ (we recall that $y$ is fixed). By monotonicity of the expectation, 
\begin{align*}
\mathbf{E}|\langle \mathcal{E}^*V_{\omega} \mathcal{E}g,g'\rangle|\leq 
\sum_{k,k'\in \Z_+}\int \mathbf{E}\max_{\mathcal{F}_k\times \mathcal{F}_{k'}}|X_{\xi,\xi'}|\rd y\rd\tau\rd\tau',
\end{align*}
where (suppressing the dependence on $y,\tau,\tau'$)
\begin{align*}
X_{\xi,\xi'}=\sum_{j\in h\Z^d}\omega_j\sum_i \overline{V(x_i+y)\xi_i}\xi_i'.
\end{align*}
The conclusion follows by Lemma \ref{lemma Dudley} and \ref{lemma Geometric series estimate} (details of the calculation are provided in the appendix).
\end{proof}

\begin{lemma}\label{lemma Dudley}
Let $R,h\geq 2$. Then we have the following bounds,
\begin{align*}
&\int \mathbf{E}\max_{\mathcal{F}_k\times \mathcal{F}_{k'}}|X_{\xi,\xi'}|\rd y\rd\tau\rd\tau'
\lesssim (\log R)^{1/2}h^{d/2}
\|V\|_{L^{2q}(\R^d)},\\
&\int \max_{\mathcal{F}_k\times \mathcal{F}_{k'}}|X_{\xi,\xi'}|\rd y\rd\tau\rd\tau'\lesssim R^{d-d/2q} 2^{-k-k'}
\|V\|_{L^{2q}(\R^d)}.
\end{align*} 
\end{lemma}

\begin{proof}
Note first that the index set of $X_{\xi,\xi'}$ is finite and has cardinality $N$, satisfying
\begin{align}\label{N=card Ftimes F'}
\log N=\log |\mathcal{F}_k\times \mathcal{F}_{k'}|\lesssim \log R\max(4^k,4^{k'}),
\end{align}
by Corollary \ref{cor. representaion} (a). 
Proposition \ref{prop. properties of subgaussian rv} implies that $X_{\xi,\xi'}$ are (scalar) subgaussian random variables, and  
\begin{align*}
\mathbf{E}\max_{\mathcal{F}_k\times \mathcal{F}_{k'}}|X_{\xi,\xi'}|
\lesssim \sqrt{\log N} (\sum_{j\in h\Z^d}|\sum_i \overline{V(x_i+y)\xi_i}\xi_i'|^2)^{1/2},
\end{align*}
where we recall that we are assuming $\|\omega_j\|_{\psi_2}\lesssim 1$. Using Hölder's inequality twice, it follows that
\begin{align}\label{Emax lemma Dudley}
\mathbf{E}\max_{\mathcal{F}_k\times \mathcal{F}_{k'}}|X_{\xi,\xi'}|
&\lesssim \sqrt{\log N}\left\|\|V(x_i+y)\|_{\ell^q_i}\|\xi_i\|_{\ell^{p'}_i}\|\xi_i'\|_{\ell^{p'}_i}\right\|_{\ell^{2}_j}\\
&\lesssim \sqrt{\log N}\left\|\|V(x_i+y)\|_{\ell^q_i}\right\|_{\ell^{2q}_j}
\left\|\|\xi_i\|_{\ell^{p'}_i}
\|\xi_i'\|_{\ell^{p'}_i}\right\|_{\ell^{p'}_j},
\end{align}
where we recall that $i$ is restricted to $x_i+y\in Q_h+j$ and $y$ is fixed. In particular, we have
\begin{align}\label{cardinality of j}
|\{j\in h\Z^d:\, x_i+y\in Q_h+j\}|&=1 \quad\mbox{for each } i
\end{align}
and 
\begin{align}\label{cardinality of i}
|\{i:\, x_i+y\in Q_h+j\}|&\leq h^d\quad\mbox{for each } j\in h\Z^d.
\end{align}
We will show that
\begin{align}\label{claim lemma Dudley}
\left\|\|\xi_i\|_{\ell^{p'}_i}
\|\xi_i'\|_{\ell^{p'}_i}\right\|_{\ell^{p'}_j}\lesssim 
h^{d/p'}\min(2^{-k},2^{-k'})
\|g(\eta_{\nu}+\tau)\|_{\ell^2_{\nu,{\rm av}}}\|g'(\eta_{\nu'}+\tau')\|_{\ell^2_{\nu',{\rm av}}}
\end{align}
By symmetry in $\xi,\xi'$, it suffices to prove this in the case $k\geq k'$. Using Hölder once more, we have
\begin{align*}
\left\|\|\xi_i\|_{\ell^{p'}_i}
\|\xi_i'\|_{\ell^{p'}_i}\right\|_{\ell^{p'}_j}\leq 
\left\|\|\xi_i\|_{\ell^{p'}_i}\right\|_{\ell^{\infty}_j}\left\|\|\xi_i'\|_{\ell^{p'}_i}\right\|_{\ell^{p'}_j}.
\end{align*}
By Fubini's theorem and \eqref{cardinality of j}, 
\begin{align*}
\left\|\|\xi_i'\|_{\ell^{p'}_i}\right\|_{\ell^{p'}_j}&=(\sum_i\sum_{j\in h\Z^d}|\xi_i'|^{p'})^{1/p'}
=(\sum_i |\xi_i'|^{p'})^{1/p'}
=\|\xi'\|_{p'}.
\end{align*}
Similarly, by \eqref{cardinality of i} we have 
\begin{align*}
\left\|\|\xi_i\|_{\ell^{p'}_i}\right\|_{\ell^{\infty}_j}\leq h^{d/p'} \|\xi\|_{\infty}.
\end{align*}
Combining these estimates with \eqref{bounds xi in terms of g}, \eqref{bounds xi' in terms of g'} yields \eqref{claim lemma Dudley}. Next, we have (again by Hölder, Fubini and \eqref{cardinality of i})
\begin{align}\label{V mixed norm Hölder, Fubini}
\left\|\|V(x_i+y)\|_{\ell^q_i}\right\|_{\ell^{2q}_j}&\leq 
 h^{d/2q}\left\|\|V(x_i+y)\|_{\ell^{2q}_i}\right\|_{\ell^{2q}_j}\\
&=
 h^{d/2q}\left\|\|V(x_i+y)\|_{\ell^{2q}_j}\right\|_{\ell^{2q}_i}\\
&= h^{d/2q}\|V(x_i+y)\|_{\ell^{2q}_i}.
\end{align}
Integrating \eqref{Emax lemma Dudley} over $y,\tau,\tau'$ and using \eqref{claim lemma Dudley}, \eqref{V mixed norm Hölder, Fubini}, we obtain 
\begin{align*}
\int \mathbf{E}\max_{\mathcal{F}_k\times \mathcal{F}_{k'}}|X_{\xi,\xi'}|\rd y\rd\tau\rd\tau'\lesssim \sqrt{\log N} \min(2^{-k},2^{-k'}) h^{d/2}
\|V\|_{L^{2q}(\R^d)},
\end{align*} 
where we used that 
$\frac{1}{2}=\frac{1}{2q}+\frac{1}{p'}$, 
$\|V(x_i+y)\|_{L^{2q}_y\ell^{2q}_i}\lesssim \|V\|_{L^{2q}(\R^d)}$ and
\begin{align*}
R^{-(d-1)}\|g(\eta_{\nu}+\tau)\|_{L^2_{\tau}\ell^2_{\nu,{\rm av}}}\|g'(\eta_{\nu'}+\tau')\|_{L^2_{\tau'}\ell^2_{\nu',{\rm av}}}\lesssim \|g\|_{L^2(M,\rd\sigma)}\|g'\|_{L^2(M,\rd\sigma)}=1.
\end{align*}
Combining this with \eqref{N=card Ftimes F'} yields the first bound of the lemma.
The second bound follows from the estimate
\begin{align*}
|X_{\xi,\xi'}|&\leq\sum_{j\in h\Z^d}|\sum_i \overline{V(x_i+y)\xi_i}\xi_i'|\leq \sum_{j\in h\Z^d}\|V(x_i+y)\|_{\ell^1_i}\|\xi_i\|_{\ell^{\infty}_i}\|\xi_i'\|_{\ell^{\infty}_i}\\
&\leq \|V(x_i+y)\|_{\ell^1_j\ell^1_i}\|\xi\|_{\infty}\|\xi'\|_{\infty}\\
&= \|V(x_i+y)\|_{\ell^1_i\ell^1_j}\|\xi\|_{\infty}\|\xi'\|_{\infty}\\
&=\|V(x_i+y)\|_{\ell^1_i}\|\xi\|_{\infty}\|\xi'\|_{\infty}\\
&\lesssim R^{d-d/2q}\|V(x_i+y)\|_{\ell^{2q}_i}\|\xi\|_{\infty}\|\xi'\|_{\infty}\\
&\lesssim R^{d-d/2q}\|V(x_i+y)\|_{\ell^{2q}_i} 2^{-k-k'}\|g(\eta_{\nu}+\tau)\|_{\ell^2_{\nu,{\rm av}}}\|g'(\eta_{\nu'}+\tau')\|_{\ell^2_{\nu',{\rm av}}},
\end{align*}
where we used Hölder in the first, second and fifth line, Fubini in the third line, \eqref{cardinality of j} in the fourth, $\supp V\subset B_R$ in the fifth and \eqref{bounds xi in terms of g}, \eqref{bounds xi' in terms of g'} in the last line. Integrating over $y,\tau,\tau'$ and using Hölder as before yields the second bound in the lemma.
\end{proof}

\begin{remark}\label{remark local uniformity}
If we restore the frequency in the extension operator, i.e.\ if we consider $\mathcal{E}_{\lambda}^*V_{\omega} \mathcal{E}_{\lambda'}$, then it is obvious from the proof of Lemma \ref{lemma local extension bound} that the same estimate holds for this operator, locally uniformly in $\lambda,\lambda'\asymp 1$. Explicitly,
\begin{align}
&\sup_{\lambda,\lambda'\asymp 1}\mathbf{E}\|\mathcal{E}_{\lambda}^*V_{\omega} \mathcal{E}_{\lambda'}\|_{L^2(M_{\lambda},\rd\sigma_{\lambda})\to L^2(M_{\lambda'},\rd\sigma_{\lambda'})}\leq A(h,R,V),\label{extension bound A(h,R,V)}\\
&A(h,R,V)\lesssim \langle h\rangle^{d/2}(\log \langle  R\rangle)^{1/2}(\log \langle h\rangle + \log \langle R\rangle)^{2}\|V\|_{L^{2q}(\R^d)}.
\end{align}
\end{remark}

\subsection{Smoothing}\label{subsection smoothing}
We observe that if $m(D)$ is a Fourier multiplier and $B_{R_1}$, $B_{R_2}$ are two balls with the same center, then
\begin{align}\label{smoothing identity}
\mathbf{1}_{B_{R_1}}m(D)\mathbf{1}_{B_{R_2}}=\mathbf{1}_{B_{R_1}}m_R(D)\mathbf{1}_{B_{R_2}},\quad m_R:=\gamma_R\ast m,
\end{align}
whenever $R>R_1+R_2$, $\gamma_R(\xi)=R^d\gamma(R\xi)$ and $(\gamma)^{\vee}$ is a bump function such that $(\gamma)^{\vee}(x)=1$ for $|x|\leq 1$. This can be checked by comparing the kernels of both sides in \eqref{smoothing identity} and using the convolution theorem. The convolution with $\gamma_R$ can be considered a smoothing operator at scale $R^{-1}$.
We recall from Section \ref{section Born series} that $C^{(\delta)}$ denotes a generic function satisfying a bound
\begin{align}\label{Cdelta bound repeated}
|C^{(\delta)}(\xi)|\lesssim (||2\pi \xi|^2-1|+\delta)^{-1/2}.
\end{align}
We will apply \eqref{smoothing identity} to 
\begin{align}\label{def. m(xi)}
m(\xi)=(|2\pi\xi|^2-(1+\I 0)^2)^{-1}
\end{align}
to produce a product of two functions $C^{(\delta)}(\xi)$ satisfying \eqref{Cdelta bound repeated} with $\delta=R^{-1}$.

\begin{lemma}\label{lemma smoothing}
For $R\geq 1$ we have
\begin{align*}
|\gamma_R\ast m|\lesssim R.
\end{align*}
In particular, $(\gamma_R\ast m)^{1/2}$ satisfies \eqref{Cdelta bound repeated} with $\delta=R^{-1}$.
\end{lemma}

\begin{proof}
By a partition of unity we may assume that $m$ is supported in a small conic neighborhood of the first coordinate axis. The implicit function theorem then allows us to reduce the proof to the following bound,
\begin{align*}
|\gamma_R\ast \frac{1}{\xi_1+\I 0}|\lesssim R,
\end{align*}
where $\gamma_R(\xi_1)=R\gamma(R\xi_1)$ is a function of one variable. By the convolution theorem, 
\begin{align*}
|\gamma_R\ast \frac{1}{\xi_1+\I 0}|\lesssim \|\hat{\gamma}_R\|_1\lesssim R,
\end{align*}
where we used that the Fourier transform of $(\xi_1+\I 0)^{-1}$ is bounded. See also \cite[Lemma 5.2]{Ruiz_lecturenotes} for an alternative proof.
\end{proof}

\begin{remark}
The boundary value in \eqref{def. m(xi)} is defined in the usual way (in the sense of tempered distributions, see e.g.\ \cite{MR1065993}). The analogue expression with $(1-\I 0)^2$ clearly satisfies the same bound. A similar argument (using the Malgrange preparation theorem) also works for $\epsilon$ nonzero and fixed. This argument is presented in the proof of Lemma 23 in \cite{bogli2021counterexample}. Alternatively, one can work with the boundary values throughout and appeal to the Phragm\'en-Lindel\"of maximum principle to extend the results to nonzero $\epsilon$ (see e.g.\ \cite[Appendix A]{MR3608659}, \cite{MR4150258}, \cite{Ruiz2002}). We will not pursue this issue.
\end{remark}

In practice, we are working with a localized version  of \eqref{def. m(xi)}, supported near the singular manifold $M$. Even though $\gamma_R\ast m$ loses compact support, it decays rapidly away from $M$ on the $1/R$ scale. Neglecting the tail (which can be bounded in a straightforward way), we assume that all functions $C^{(\delta)}$ that appear from now on are compactly supported in a small neighborhood of $M$. Alternatively, one could avoid tails by smoothing the resolvent first and then perform the low/high decomposition as in Section \ref{section Born series}.

\subsection{Foliation by level sets}
In the following we will assume that $C^{(\delta)}$ is supported in a $c$-neighborhood ($c$ small and fixed) of $M$ and satisfies \eqref{Cdelta bound repeated}. We will also assume that $\lambda\in [1-c,1+c]$ and denote the constant $A(h,R,V)$ appearing in \eqref{extension bound A(h,R,V)} by $A$.

\begin{lemma}\label{lemma foliation}
Assume that \eqref{extension bound A(h,R,V)}, \eqref{Cdelta bound repeated} hold. Then we have 
\begin{align}\label{bound 2}
\|\mathcal{E}^*_{\lambda} VC^{(\delta)}\|_{L^2(\R^d)\to L^2(M_{\lambda})}\lesssim A(\log\frac{1}{\delta})^{\frac{1}{2}}.
\end{align}
Moreover, if \eqref{Cdelta bound repeated} holds for $C^{(\delta_1)}$, $C^{(\delta_2)}$, then
\begin{align}\label{bound 3}
\|C^{(\delta_1)}VC^{(\delta_2)}\|_{L^2(\R^d)\to L^2(\R^d)}\lesssim A(\log\frac{1}{\delta_1})^{\frac{1}{2}}(\log\frac{1}{\delta_2})^{\frac{1}{2}}.
\end{align}
\end{lemma}

\begin{proof}
For $f\in L^2(\R^d)$ we foliate by level sets $M_{\lambda}$,
\begin{align}\label{foliation}
C^{(\delta)}f(x)
=\int_{1-c}^{1+c} \int_{M_{\lambda'}}\e(x\cdot\xi)C^{(\delta)}(\xi)\widehat{f}(\xi)\rd \sigma_{\lambda'}(\xi)\rd \lambda',
\end{align}
up to an innocuous Jacobian factor. Using \eqref{Cdelta bound repeated} and the fact that $(\rd \sigma_{\lambda})^{\vee}\ast f$ is a constant multiple of $\mathcal{E}_{\lambda}\mathcal{E}^*_{\lambda}f$, we get, by Cauchy--Schwarz,
\begin{align*}
\|\mathcal{E}^*_{\lambda} VC^{(\delta)}f\|_{L^2(M)}&
\leq A (\int_{1-c}^{1+c}\rd \lambda'(|\lambda'-1|+\delta)^{-1})^{1/2}(\int_{1-c}^{1+c}\rd \lambda' \|\mathcal{E}_{\lambda'}f\|^2_{L^2(M_{\lambda'})})^{1/2}\\
&\lesssim A(\log\frac{1}{\delta})^{\frac{1}{2}}\|f\|_{2}.
\end{align*}
where we used 
\begin{align}\label{bound 4}
\int_{1-c}^{1+c}\rd \lambda' \|\mathcal{E}_{\lambda'}^*f\|^2_{L^2(M_{\lambda'})}=\int_{1-c}^{1+c}\rd \lambda' \int_{M_{\lambda'}}|\widehat{f}(\xi)|^2\rd \sigma_{\lambda'}(\xi)\lesssim \|f\|_{L^2(\R^d)}^2
\end{align}
and 
\begin{align}\label{log}
\int_{1-c}^{1+c}\rd \lambda'(|\lambda'-1|+\delta)^{-1}\lesssim \log\frac{1}{\delta}.
\end{align}
This proves \eqref{bound 2}.
To prove \eqref{bound 3} we use the dual estimate to \eqref{bound 4}, which is
\begin{align}\label{bound 5}
\|\int_{1-c}^{1+c} \mathcal{E}_{\lambda'}g(\lambda')\rd \lambda'\|_{L^2(\R^d)}\lesssim (\int_{1-c}^{1+c} \|g(\lambda')\|^2_{L^2(M_{\lambda'})}\rd \lambda')^{1/2}
\end{align}
for $g(\lambda')\in L^2(M_{\lambda'})$. This follows from
\begin{align*}
\int_{1-c}^{1+c}\langle \mathcal{E}_{\lambda'}^* f,g(\lambda')\rangle_{L^2(M_{\lambda'})}\rd \lambda'=\langle f,\int_{1-c}^{1+c} \mathcal{E}_{\lambda'}g(\lambda')\rd \lambda'\rangle_{L^2(\R^d)}.
\end{align*}
Using the foliation \eqref{foliation} for the $C^{(\delta_1)}$ factor and using \eqref{bound 2}, \eqref{log}, \eqref{bound 5} gives, with $g(\lambda')=(|\lambda'-1|+\delta_1)^{-1/2}\mathcal{E}_{\lambda'}^*VC^{(\delta_2)}f$, 
\begin{align*}
\|C^{(\delta_1)}VC^{(\delta_2)}f\|_{L^2(\R^d)}&\lesssim
\|\int_{1-c}^{1+c}\mathcal{E}_{\lambda'}g(\lambda')\rd \lambda'\|_{L^2(\R^d)}\lesssim(\int_{1-c}^{1+c} \|g(\lambda')\|^2_{L^2(M_{\lambda'})}\rd \lambda')^{1/2}\\
&\lesssim  A(\log\frac{1}{\delta_1})^{\frac{1}{2}}(\log\frac{1}{\delta_2})^{\frac{1}{2}}\|f\|_{L^2(\R^d)}.
\end{align*}
\end{proof}

\subsection{Local resolvent bound}
We use the same conventions as in the previous section. Additionally, in the following, the norm is the $L^2(\R^d)\to L^2(\R^d)$ operator norm. Recall that, by the discussion at the end of Section \ref{subsection smoothing}, the square root of the localized resolvent $R_0^{\rm low}$ can be replaced by a compactly supported multiplier satisfying the bound \eqref{Cdelta bound repeated} with $\delta=1/R$.
As a consequence of Lemma \ref{lemma local extension bound}, \eqref{bound 3} and the discussion in Section \ref{section Born series}, we immediately obtain the following resolvent bound.

\begin{lemma}\label{lemma local resolvent bound}
Assume that \eqref{Cdelta bound repeated} holds for 
$C^{(\delta_1)}$, $C^{(\delta_2)}$, with $\delta_1,\delta_2\asymp 1/R$. Then we have
\begin{align*}
\mathbf{E}\|C^{(\delta_2)}V_{\omega}C^{(\delta_1)}\|\lesssim \langle  h\rangle^{d/2}(\log \langle  R\rangle)^{3/2}(\log \langle  h\rangle + \log \langle  R\rangle)^{2}\|V\|_{L^{2q}(\R^d)}.
\end{align*}
\end{lemma}

By using the tail bound of Lemma \ref{lemma tail bound} and rescaling, we obtain the following corollary.

\begin{corollary}\label{cor. spectral radius}
Let $h,R,\lambda,M>0$ and let $|\epsilon|\ll \lambda$. Then the spectral radius of $R_0((\lambda+\I\epsilon)^2)V_{\omega}$ is bounded by
\begin{align*}
\spr(R_0V)\lesssim M\langle \lambda h\rangle^{d/2}(\log \langle \lambda R\rangle)^{3/2}(\log \langle \lambda h\rangle + \log \langle\lambda R\rangle)^{2}\lambda^{\frac{d}{2q}-2}\|V\|_{L^{2q}(\R^d)},
\end{align*}
except for $\omega$ in a set of measure at most
$\exp(-cM^2)$.
\end{corollary}

\subsection{Completion of the proof of Theorem \ref{thm. 1}}
We first undo the change of variables $q\to 2q$.
Theorem \ref{thm. 1} then follows immediately from Proposition \ref{prop. spectral radius <1} and Corollary \ref{cor. spectral radius}. \qed

\section{Local to global arguments}\label{section Local to global arguments}

\subsection{Proof of Theorem \ref{thm. 2}}
To complete the proof of Theorem \ref{thm. 2} we rescale again to $\lambda=1$. We
decompose $V=\sum_{k\in\Z^+} V_k$ into dyadic pieces with support in $\{0\leq |x|\leq 1\}$ for $k=0$ and in $\{2^{k-1}\leq |x|\leq 2^{k}\}$ for $k\geq 1$. The assumption on $V$ guarantees that $\|V_k\|_{q}\leq 2^{-\delta k}\|\langle x\rangle^{\delta}V\|_q$. 
Instead of \eqref{multilinear expansion Born series}, we consider the multilinear expansion 
\begin{align*}
[R_0V]^n=\sum_{\sigma_1,\ldots,\sigma_n}\sum_{k_1,\ldots,k_n}R_0^{\sigma_1}V_{k_1}R_0^{\sigma_2}V_{k_2}\ldots R_0^{\sigma_n}V_{k_n},
\end{align*}
where we again omitted the spectral parameter $z$, and we are assuming, as we may, that $z=(1+\I\epsilon)^2$, $|\epsilon|\ll 1$.
kBy the same arguments as in Section \ref{section Born series} it suffices to estimate the norms of elementary blocks of the form $C^{(\delta_{l-1})}V_{k_l}C^{(\delta_l)}$, where $\delta_l=(2^{k_l}+2^{k_{l-1}})^{-1}$.
Lemmas \ref{lemma local extension bound}, \ref{lemma foliation} and an analogue of Lemma \ref{lemma smoothing} with $\delta=\delta_l$ or $\delta_{l-1}$ yield that (again undoing the change of variables $q\to 2q$) 
\begin{align*}
\mathbf{E}\|C^{(\delta_{l-1})}V_{k_l}C^{(\delta_l)}\|\lesssim (k_{l-1}+k_l+k_{l+1}) \langle  h\rangle^{d/2}\langle k_l\rangle^{1/2}(\log \langle  h\rangle + \langle k_l\rangle)^{2}2^{-\delta k_l}\|\langle x\rangle^{\delta}V\|_{q}.
\end{align*}
Applying the tail bound of Lemma \ref{lemma tail bound} 
yields that
\begin{align*}
\|C^{(\delta_{l-1})}V_{k_l}C^{(\delta_l)}\|\leq M_1 (k_{l-1}+k_l+k_{l+1}) \langle  h\rangle^{d/2}\langle k_l\rangle^{1/2}(\log \langle  h\rangle + \langle k_l\rangle)^{2}2^{-\delta k_l}\|\langle x\rangle^{\delta}V\|_{q},
\end{align*}
except for $\omega$ in a set of measure at most $\exp(-c'M_1^2)$. Choosing $M_1=M(k_{l-1}+k_l+k_{l+1})$ and summing the previous bound over $k_1,\ldots,k_n$ yields
\begin{align*}
\spr(R_0V)=\lim_{n\to\infty}\|[R_0V]^n\|^{1/n}\lesssim \langle  h\rangle^{d/2}(\log \langle  h\rangle)^{2}\|\langle x\rangle^{\delta}V\|_{q},
\end{align*}
except for $\omega$ in a set of measure at most
\begin{align*}
\sum_{k_{l-1},k_l,k_{l+1}}\exp(-c'M_1^2)\leq \exp(-cM^2).
\end{align*}
This concludes the proof of Theorem \ref{thm. 2}. \qed

\subsection{Sparse decomposition}\label{subsect. sparse decomp.}
To prove Theorem \ref{thm. 3} we use a device reminiscent of an ``epsilon removal lemma" (see e.g.\ \cite{MR1666558}) but adapted to our multilinear bounds (and the resolvent as opposed to the Fourier restriction operator). For this reason, we need to perform several decompositions simultaneously:

\begin{enumerate}
\item We first decompose $V$ dyadically:
\begin{align*}
V=\sum_{i\in\Z_+}V_i,\quad V_i=V\mathbf{1}_{H_i\geq |V|\geq H_{i+1}},\quad H_i=\inf\{t>0:\,|\{|V|>t\}|\leq 2^{i-1}\}.
\end{align*}
This is a ``horizontal" dyadic decomposition since the widths of the supports of $V_i$ are approximately $2^i$.
Here we are assuming that $V$ is constant on the unit scale (hence $i\geq 0$ in the sum above). In view of \eqref{smoothing identity for V}, there is no loss of generality in this assumption for the purpose of proving estimates (this is the same argument as explained in the paragraph before  \cite[Lemma 3.3]{MR1666558}).
Note that we have 
\begin{align*}
\|H_i2^{i/q}\|_{\ell^r_i(\Z_+)}\asymp \|V\|_{L^{q,r}},
\end{align*}
where $L^{q,r}$ denotes a Lorentz space (see e.g.\ \cite[Thm. 6.6]{TaoNotes1247A}). Also note that $L^{q,q}=L^q$.
\item Next, split each dyadic piece into a sum of ``sparse families",
\begin{align}\label{sparse decomp.}
V_i=\sum_{j=1}^{K_i}\sum_{k=1}^{N_i}V_{ijk},
\end{align}
where, for fixed $i,j$, the $V_{ijk}$ are supported on a ``sparse collection" of balls $\{B(x_k,R_i)\}_{k=1}^{N_i}$. By this we mean that the support of $V_{ijk}$ is contained in $B(x_k,R_i)$ and that the following definition is satisfied (cf.\ \cite[Def. 3.1]{MR1666558}) for some sufficiently large $\gamma$ (to be chosen later):
\end{enumerate}

\begin{definition}
A collection $\{B(x_k,R)\}_{k=1}^N$ is $\gamma$-sparse if the centers $x_k$ are $(RN)^{\gamma}$ separated.
\end{definition}

For fixed $\gamma>0$ and $K>0$, \cite[Lemma 3.3]{MR1666558} asserts that \eqref{sparse decomp.} holds with 
\begin{align}\label{KiNiRi}
K_i=\mathcal{O}(K2^{i/K}),\quad N_i=\mathcal{O}(2^i),\quad R_i=\mathcal{O}(2^{i\gamma^K}).
\end{align}

\subsection{Spectral radius estimates}\label{subsection Spectral radius estimates}
The preceding decompositions produce a multilinear expansion of the Born series,
 \begin{align}\label{Born series alpha}
 [R_0V]^n=\sum_{\alpha_1,\ldots,\alpha_n}R_0V_{\alpha_1}R_0V_{\alpha_2}\ldots R_0V_{\alpha_n},
 \end{align}
where $\alpha_l=(i_l,j_l,k_l)$ and $i_l\in\Z_+$, $1\leq j_l\leq K_{i_l}$, $1\leq k_l\leq N_{i_l}$. To estimate the spectral radius of $R_0V$, we estimate the summands in~\eqref{Born series alpha} in two different ways. For the first estimate, we follow a similar strategy as before. However, since the smoothing of the resolvent (see Subsection \ref{subsection smoothing}) now depends on the mutual positions of the supports of $V_{\alpha_l}$, we consider the following (slightly more general) elementary operators,
\begin{align}\label{elementary local to global}
\mathbf{1}_{B_1}C^{(\delta_1)}\mathbf{1}_{B_2}WC^{(\delta_2)}\mathbf{1}_{B_3},
\end{align}
where $B_k=B(x_k,R_k)$ are arbitrary balls and $W$ is a bounded potential. As before, $C^{(\delta)}$ are Fourier multipliers satisfying \eqref{Cdelta bound repeated}, now with
\begin{align*}
\delta_1=\langle\rd(B_1,B_2)+2R_1+2R_2\rangle^{-1},\quad \delta_2=\langle\rd(B_2,B_3)+2R_2+2R_3\rangle^{-1}.
\end{align*}
The operators \eqref{elementary local to global} arise from an analogue of \eqref{smoothing identity} and Lemma \ref{lemma smoothing} for balls with different centers. In the same way that Lemma~\ref{lemma local resolvent bound} and its corollary follow from 
Lemma \ref{lemma local extension bound}, \eqref{bound 3}
and the tail bound of Lemma \ref{lemma tail bound}, we obtain 
\begin{align}\label{random sparse bound}
\|\mathbf{1}_{B_1}C^{(\delta_1)}\mathbf{1}_{B_2}W_{\omega}C^{(\delta_2)}\mathbf{1}_{B_3}\|\leq M_1h^{\frac{d}{2}}(\log h)^2[\log(\frac{1}{\delta_1}+\frac{1}{\delta_2})]^{\mathcal{O}(1)}\|W\|_{L^q(B_2)}
\end{align}
for any $q\leq d+1$ and for all $\omega$ except for a set of measure at most $\exp(-c' M_1^2)$. Here we have assumed again, as we may, that $\lambda=1$, $R,h>2$. For the remainder of this section we omit the (obvious) dependence on $h$. We also switch from the (modified)
Vinogradov notation $A\lesssim B$ to the Hardy notation $A\leq C B$ or Landau notation $A=\mathcal{O}(B)$, and we indicate the dependence of constants on $q$ (since $q$ will no longer be in a compact interval) or other related parameters. It is also convenient to use the letter $A$ for quantities (norms, constants) containing $\mathcal{O}(1)$ terms that are bounded uniformly in $n$ (and may change from line to line).

The case of interest is of course when the balls in \eqref{random sparse bound} contain the supports of the potentials in \eqref{Born series alpha} and $W$ is one of these potentials. Similarly as in the proof of Theorem \ref{thm. 2}, we may choose $M_1=M[\log(1/\delta_1+1/\delta_2)]^{\mathcal{O}(1)}$ without qualitatively changing the estimate \eqref{random sparse bound}. In this way,
the union bound for the probability of the complementary event yields
\begin{align*}
&\mathbf{P}(\bigcup_{\alpha_1,\alpha_2,\alpha_3}\{\omega:\, \mbox{\eqref{random sparse bound} does not hold}\})\leq \sum_{\alpha_1,\alpha_2,\alpha_3}\exp(-c' M_1^2)\\
&\leq  \sum_{i_1,i_2,i_3}N_{i_1}K_{i_1}N_{i_2}K_{i_2}N_{i_3}K_{i_3}\exp(-c' M_1^2)
\leq \exp(-c M^2),
\end{align*}
and hence we have 
\begin{align}\label{random sparse bound 2}
\|R_0V_{\alpha_1}R_0V_{\alpha_2}\ldots R_0V_{\alpha_n}\|\leq A M^n\prod_{l=1}^n[\log(1/\delta_{\alpha_l})+\log(1/\delta_{\alpha_{l+1}})]^{\mathcal{O}(1)}\|V_{\alpha_l}\|_{q},
\end{align}
except for $\omega$ in a set of measure at most $\exp(-c M^2)$.

For the second estimate, we observe that, by the triangle inequality and Cauchy--Schwarz,
\begin{align}\label{multilnear sparse 2}
\|[R_0V]^n\|\leq\sum_{\alpha_1,\ldots,\alpha_n}\|R_0|V_{\alpha_1}|^{\frac 12}\| \|V_{\alpha_1}^{\frac 12}R_0|V_{\alpha_2}|^{\frac 12}\|\ldots \|V_{\alpha_{n-1}}^{\frac 12}R_0|V_{\alpha_n}|^{\frac 12}\| \|V_{\alpha_n}^{\frac 12}\|.
\end{align}
Here we are again assuming, as we may, that $V$ is bounded. The operator norm $\|V_{\alpha_n}^{\frac 12}\|$ (equal to the $L^{\infty}$ norm) will be annihilated by taking the $n$-th root at the end and letting $n$ tend to infinity. Let
\begin{align*}
L_{\alpha,\beta}:=\delta_{\alpha,\beta}+\rd(B_{\alpha},B_{\beta}),
\end{align*}
where the balls $B_{\alpha}$ contain the support of $V_{\alpha}$.

\begin{lemma}\label{lemma BSij bound}
For $q\leq (d+1)/2$,
\begin{align}\label{BSij bound}
\|V_{\alpha}^{\frac 12}R_0|V_{\beta}|^{\frac 12}\|\leq C_qL_{\alpha,\beta}^{1-\frac{d+1}{2q}}\|V_{\alpha}\|_{q}^{1/2}\|V_{\beta}\|_{q}^{1/2}.
\end{align}
\end{lemma}
\begin{proof}
To prove this, one uses the well known pointwise bound
\begin{align}\label{pointwise bound on complex power of the resolvent}
|R_0^{(a+\I t)}(x-y)|\leq C_1\e^{C_2t^2}|x-y|^{-\frac{d+1}{2}+a}
\end{align}
for $a\in [(d-1)/2,(d+1)/2]$ and $d\geq 2$ (see e.g.\ \cite[(2.5)]{MR3865141}), or the explicit formula for the resolvent kernel in $d=1$.  More precisely, consider the analytic family $V_{\alpha}^{\zeta/2}R_0^{\zeta}|V_{\beta}|^{\zeta/2}$. Then \eqref{pointwise bound on complex power of the resolvent} implies that, for $\re\zeta=q$, the kernel is bounded by
\begin{align*}
|V_{\alpha}(x)^{\zeta/2}R_0^{\zeta}(x-y)|V_{\beta}(y)|^{\zeta/2}|\leq C_1\e^{C_2(\im\zeta)^2} L_{\alpha,\beta}^{-\eta}|V_{\alpha}(x)|^{q}|V_{\beta}(y)|^{q},
\end{align*}
where $\eta=(d+1)/2-q\geq 0$,
leading to the Hilbert--Schmidt bound
\begin{align}\label{Hilbert--Schmidt bound}
\|V_{\alpha}^{\zeta/2}R_0^{\zeta}|V_{\beta}|^{\zeta/2}\|\leq C_{\eta} L_{\alpha,\beta}^{-\eta}\|V_{\alpha}\|_{q}^{q/2}\|V_{\beta}\|_{q}^{q/2}
\end{align}
for some constant $C_{\eta}$ (allowed to change from line to line).
Interpolating this with the trivial bound $\|V_{\alpha}^{\zeta/2}R_0^{\zeta}|V_{\beta}|^{\zeta/2}\|\leq C_1\e^{C_2(\im\zeta)^2}$ for $\re\zeta=0$ yields \eqref{BSij bound}.
\end{proof}
The previous lemma yields the second estimate
\begin{align*}
\|R_0V_{\alpha_1}R_0V_{\alpha_2}\ldots R_0V_{\alpha_n}\|\leq 
AC_{\eta}^{n}\prod_{l=1}^n\|V_{\alpha_l}\|_{q_{\eta}}
L_{\alpha_{l},\alpha_{l+1}}^{-\eta'}
\end{align*}
where $\eta'=\eta/((d+1)/2-\eta)$ and $q_{\eta}=(d+1)/2-\eta$.
Interpolating this with \eqref{random sparse bound 2}, we get, for $0<\theta< 1$, 
\begin{align*}
\|R_0V_{\alpha_1}R_0V_{\alpha_2}\ldots R_0V_{\alpha_n}&\|\leq A(C_{\eta}M)^n
\prod_{l=1}^n[\log(1+R_{i_{l-1}}+R_{i_{l}}+R_{i_{l+1}})]^{\mathcal{O}(1)}
L_{\alpha_l,\alpha_{l+1}}^{-\theta\eta'/2}\\
&\times
\|V_{\alpha_l}\|_{q}^{(1-\theta)}\|V_{\alpha_l}\|_{q_{\eta}}^{\theta}.
\end{align*}
except on an exceptional set of measure at most $\exp(-cM^2)$.
(Here we used $L_{\alpha_l,\alpha_{l+1}}^{-\theta\eta'/2}$ to control
$\rd(B_{\alpha_l},B_{\alpha_{l+1}})$ appearing in $\log(1/\delta_{\alpha_l})$.)
Using that
\begin{align*}
\|V_{\alpha_l}\|_{q}\lesssim H_{i_l}2^{i_l/q}
\end{align*}
for all $q\geq 1$, and summing the resulting estimate first over $k_1$, then continuing up to $k_{n-1}$, yields
\begin{align*}
\sum_{k_1,\ldots,k_{n-1}}\|R_0V_{\alpha_1}R_0V_{\alpha_2}\ldots R_0V_{\alpha_n}&\|\leq A(C_{\eta}M)^n\prod_{l=1}^{n-1}[\log(1+R_{i_{l-1}}+R_{i_{l}}+R_{i_{l+1}})]^{\mathcal{O}(1)}\\
\times H_{i_l}2^{i_l((1-\theta)/q+\theta/q_{\eta})}.
\end{align*}
Here we have used that, for $\alpha_1=(i_1,j_1,k_1)$, $\alpha_2=(i_2,j_2,k_2)$ and $i_1,j_1,i_2,j_2,k_2$ fixed, the sum over $k_1$ is bounded,
\begin{align}\label{supsum}
\sum_{k_1\leq N_{i_1}}\langle \rd(B(x_{k_1},R_{i_1}),B_{\alpha_2})\rangle^{-\theta\eta'/2}=\mathcal{O}_{\gamma_0}(1),
\end{align}
uniformly in $i_1,j_1,i_2,j_2,k_2$, provided $\theta\eta'\gamma_0/2>1$ and $\gamma\geq \gamma_0$. We will momentarily fix $\eta,\theta$, and then choose $\gamma_0=4/(\eta'\theta)$.
Note that, even though the balls in \eqref{supsum} may belong to different sparse families, we have that
\begin{align*}
\rd(B(x_{k_1},R_{i_1}),B_{\alpha_2})\geq \frac{1}{2}(N_{i_1}R_{i_1})^{\gamma}
\end{align*}
for all but at most one $k_1$. Indeed, suppose for contradiction that this does not hold for two distinct $k_1,k_1'$. Then by the triangle inequality,
\begin{align*}
\rd(B(x_{k_1},R_{i_1}),B(x_{k_1'},R_{i_1}))<(N_{i_1}R_{i_1})^{\gamma},
\end{align*}
which contradicts the sparsity of the collection $\{B(x_{k_1},R_{i_1})\}$.

Note that the last summation over $k_n$ produces a $\mathcal{O}(2^{i_n})$ factor, but this can be absorbed into the constant $A$ after summing over $i_n$ and hence we do dot display it. 

Summing over $j_1,\ldots,j_{n}$ yields
\begin{align*}
&\sum_{j_1,\ldots,j_n}\sum_{k_1,\ldots,k_n}\|R_0V_{\alpha_1}R_0V_{\alpha_2}\ldots R_0V_{\alpha_n}\|
\\&\leq  A(C_{\eta}M)^n\prod_{l=1}^n[\log(1+R_{i_{l-1}}+R_{i_{l}}+R_{i_{l+1}})]^{\mathcal{O}(1)}
 K_{i_l}H_{i_l}2^{i_l((1-\theta)/q+\theta/q_{\eta})},
\end{align*} 
where $K_i$ is as in \eqref{KiNiRi}. Finally, summing over $i_1,\ldots,i_{n}$ yields
\begin{align*}
\|[R_0V]^n\|&\leq A(C_{\eta}M K)^n(\sum_{i\in\Z_+}\langle i\rangle^{\mathcal{O}(1)}H_{i}2^{i((1-\theta)/q+\theta/q_{\eta}+1/K)})^n.
\end{align*} 
Once $K$ is fixed, we choose $\eta,\theta$ such that $0<\theta(1/q_{\eta}-1/q)<1/K$. Then
\begin{align}\label{limK}
\spr(R_0V_{\omega})=\lim_{n\to\infty}\|[R_0V]^n\|^{1/n}&\leq C_{\eta,K}M\sum_{i\in\Z_+}H_{i}2^{i/q}2^{3i/K},
\end{align} 
where we used that $\langle i\rangle^{\mathcal{O}(1)}\leq C_K 2^{i/K}$.

\subsection{Completion of the proof of Theorem \ref{thm. 3}}\label{subsection Completion of the proof of Theorem 3}
We use \eqref{limK} for $\tilde{q}>q$ instead of $q$, that is we now regard $(d+1)/2<q<d+1$ as given and choose $\tilde{q}< d+1$ and $K$ such that $1/\tilde{q}+3/K<1/q$. Then
\begin{align*}
\spr(R_0V_{\omega})\lesssim \sup_{i\in\Z^+}H_i2^{i/q} \sum_{i\in\Z_+}2^{i(1/\tilde{q}-1/q+3/K)}\leq  C_{\tilde{q},K}M \|V\|_{L^{q,\infty}}.
\end{align*}
Clearly, the choice of $\tilde{q}$ depends only on $q,K,d$ and $\|V\|_{L^{q}}\leq \|V\|_{L^{q,\infty}}$.
We have thus proved the main estimate of this section, which also completes the proof of Theorem \ref{thm. 3}.

\begin{lemma}\label{lemma completion thm 3}
Let $q<d+1$. Then there exists $c,M_0$ such that for all $M\geq M_0$, $z=(\lambda+\I\epsilon)^2$, $\lambda\asymp 1$, $|\epsilon|\ll 1$ and $V\in L^q(\R^d)$,
\begin{align*}
\spr(R_0(z)V_{\omega})\leq M\|V\|_q 
\end{align*}
except for $\omega$ in a set of measure at most
$\exp(-cM^2)$.
\end{lemma}

\subsection{Global extension bound}
For potential future reference we include a similar bound to that proved in Lemma \ref{lemma completion thm 3}, but for the norms of the elementary operators \eqref{elementary operators} instead of the spectral radius of $R_0V$.

\begin{proposition}\label{prop. Global extension bound}
Let $q< d+1$. Then there exist constants $M_0,c$ such for any $M\geq M_0$, $\lambda,\lambda'\asymp 1$ and $V\in L^q(\R^d)$,
\begin{align*}
\|\mathcal{E}_{\lambda}^*V_{\omega}\mathcal{E}_{\lambda'}\|\leq M\langle h\rangle^{d/2}(\log\langle h\rangle)^2\|V\|_{L^q},
\end{align*}
except for $\omega$ in a set of measure at most
$\exp(-cM^2)$.
\end{proposition}

In the following, we use the notation $\|V\|_{\ell^{\infty}L^{q}}=\sup_{j\leq N}\|V\|_{L^{q}(B(x_j,R))}$ and $V_j=V\mathbf{1}_{(B(x_j,R))}$, whenever $V$ is supported on a $\gamma$-sparse collection $\{B(x_j,R)\}_{j=1}^N$.
We will show that Lemma \ref{prop. Global extension bound} follows from the subsequent lemma.

\begin{lemma}\label{lemma sparse extension bound}
There exist constants $M_0,c,\gamma_0>0$ such that the following holds.
For any $R>0$, $0<h<R$, $\lambda,\lambda'\asymp 1$, $q< d+1$, $N\in\N$, $\gamma\geq \gamma_0$, for any $V\in L^q(\R^d)$ supported on a $\gamma$-sparse collection $\{B(x_j,R)\}_{j=1}^N$, and for any $M\geq M_0$, $\epsilon>0$, 
\begin{align*}
\|\mathcal{E}_{\lambda}^*V_{\omega}\mathcal{E}_{\lambda'}\|\leq C_{q,\epsilon} (M^2+\log N)^{1/2}\langle h\rangle^{d/2}(\log\langle h\rangle)^2 \langle R\rangle^{\epsilon}\|V\|_{\ell^{\infty}L^q},
\end{align*}
except for $\omega$ in a set of measure at most
$\exp(-cM^2)$.
\end{lemma}

\begin{proof}
We may assume without loss of generality that $\lambda,\lambda'=1$ and $R>2$. We omit the subscripts in $\mathcal{E}^*_{\lambda}$, $\mathcal{E}_{\lambda'}$ as well as the (obvious) $h$-dependence (i.e.\ we set $h=1$). Consider the operators
\begin{align*}
T_j=\mathcal{E}^*V_j\mathcal{E},\quad 1\leq j\leq N,
\end{align*}
where we omitted $\omega$ from the notation. Then
\begin{align*}
T_iT_j^*=\mathcal{E}^*V_i\mathcal{E}\mathcal{E}^*\overline{V_j}\mathcal{E},\quad T_i^*T_j=\mathcal{E}^*\overline{V_i}\mathcal{E}\mathcal{E}^*V_j\mathcal{E}.
\end{align*}
As in the endpoint proof of the Stein--Tomas theorem (see e.g.\ \cite[IX.1.2.2]{MR1232192}) we embed $\mathcal{E}\mathcal{E}^*$ into an analytic family of operators $U_s$ in the strip $(1-d)/2\leq \re s\leq 1$, satisfying 
\begin{align*}
\|U_s\|_{L^2\to L^2}&\lesssim 1,\quad \re s=1,\\
\|U_s\|_{L^1\to L^{\infty}}&\lesssim 1,\quad \re s=(1-d)/2,
\end{align*}
and $U_0=\mathcal{E}\mathcal{E}^*$. Similarly as in the proof of Lemma \ref{lemma BSij bound} we then use complex interpolation on the family $|V_i|^{\frac{1-s}{2}}U_s|V_j|^{\frac{1-s}{2}}$ to obtain the bound
\begin{align*}
\||V_i|^{\frac{1}{2}}\mathcal{E}\mathcal{E}^*|V_j|^{\frac{1}{2}}\|\lesssim L_{ij}^{-\eta'}\|V_i\|_{L^{q_{\eta}}}^{\frac{1}{2}}\|V_j\|_{L^{q_{\eta}}}^{\frac{1}{2}}
\end{align*}
for $\eta'=\eta/q_{\eta}$, $q_{\eta}=(d+1)/2-\eta$ and $0<\eta\ll 1$. By the Stein--Tomas and Hölder's inequality, we also have 
\begin{align*}
\|\mathcal{E}^*V_i^{\frac{1}{2}}\|\lesssim \|V_j\|_{L^{q_{\eta}}}^{\frac{1}{2}}, \quad\|V_i^{\frac{1}{2}}\mathcal{E}\|\lesssim \|V_j\|_{L^{q_{\eta}}}^{\frac{1}{2}}.
\end{align*}
Combining the last two displayed formulas yields the deterministic bound
\begin{align*}
\|T_iT_j^*\|^{\frac{1}{2}}+\|T_i^*T_j\|^{\frac{1}{2}}\lesssim L_{ij}^{-\eta'}\|V\|_{\ell^{\infty}L^{q_{\eta}}}
\end{align*}
for all $i,j\leq N$.
On the other hand, the bound of Lemma \ref{lemma local extension bound} (and changing variables $2q\to q)$ yields
\begin{align*}
\|T_iT_j^*\|^{\frac{1}{2}}+\|T_i^*T_j\|^{\frac{1}{2}}\leq  M_1(\log R)^{5/2}\|V\|_{\ell^{\infty}L^{q}}
\end{align*}
for all $i,j\leq N$, and for all $\omega$ except for an exceptional set of measure at most $N\exp(-cM_1^2)$. Interpolating the previous two estimates as in the proof of Lemma~\ref{lemma completion thm 3}, we get
by the Cotlar--Stein lemma and \eqref{supsum},
\begin{align*}
\|\mathcal{E}^*V\mathcal{E}\|\leq C_{\eta,\gamma_0}[(\log R)^{5/2}\|V\|_{\ell^{\infty}L^{q}}]^{1-\theta}\|V\|_{\ell^{\infty}L^{q_{\eta}}}^{\theta}
\end{align*}
for any $\theta\in (0,1)$ and for all $\omega$ except for an exceptional set, provided $\theta\eta'\gamma_0/2>1$ and $\gamma\geq \gamma_0$.
Finally, we use Hölder's inequality 
\begin{align*}
\|V\|_{\ell^{\infty}L^{q_{\eta}}}
\lesssim R^{d/s_{\eta}}\|V\|_{\ell^{\infty}L^{q}},\quad \frac{1}{q_{\eta}}=\frac{1}{s_{\eta}}+\frac{1}{q},
\end{align*}
to convert the previous estimate to
\begin{align}
\|\mathcal{E}^*V\mathcal{E}\|\leq C_{\eta,\gamma_0}[\log R]^{5(1-\theta)/2}R^{\theta d/s_{\eta}}\|V\|_{\ell^{\infty}L^{q}}.
\end{align}

We now fix $0<\eta\ll 1$ (small, but independent of $\epsilon$) and choose $\theta\in (0,1)$ such that
\begin{align*}
[\log R]^{5(1-\theta)/2}R^{\theta d/s}\leq R^{\epsilon}.
\end{align*}
Moreover, we choose $M_1=(M^2+c^{-1}\log N)^{1/2}$, which ensures that the exceptional set has measure at most $\exp(-cM^2)$.
Then the claim holds with the choice $\gamma_0=4/(\eta'\theta)$. 
The remainder of the proof is the same as that of Lemma~\ref{lemma completion thm 3}.
\end{proof}

\begin{proof}[Proof of Proposition \ref{prop. Global extension bound}]
We again use the sparse decomposition of Subsection \ref{subsect. sparse decomp.} and recall the bounds \eqref{KiNiRi} on $K_i,N_i,R_i$. As before, we also set $\lambda,\lambda',h=1$. Lemma \ref{lemma sparse extension bound} yields the estimate
\begin{align*}
\|\mathcal{E}^*V_{ij}\mathcal{E}\|\leq C_{q,\epsilon}(M_i^2+\log N_i)^{1/2}R_i^{\epsilon}\|V_{ij}\|_q 
\end{align*}
for all $q<d+1$, uniformly in $i,j$ and for $\omega$ outside of a set of measure $\exp(-cM_i^2)$. Here we are assuming, as we may, that $M_i,N_i,R_i>2$, say. We may choose $M_i$ freely, and we take $M_i=2M\langle i\rangle^{\delta}$, with $\delta>0$. Summing over $j$ yields, by the triangle inequality,
\begin{align*}
\|\mathcal{E}^*V_{i}\mathcal{E}\|\leq C_{q,\epsilon}K_i(M_i^2+\log N_i)^{1/2}R_i^{\epsilon}\|V_{i}\|_q. 
\end{align*}
Summing over $i$,
\begin{align*}
\|\mathcal{E}^*V_{i}\mathcal{E}\|\leq C_{q,\epsilon,K}\sum_{i\in\Z_+}
H_i2^{i(1/q+2/K+\epsilon\gamma^K)}.
\end{align*}
Here we also used \eqref{KiNiRi}, $\|V_{i}\|_q\lesssim H_i2^{i/q}$ and 
$(M^2+\log N_i)^{1/2}\leq C_{K} M 2^{i/K}$.
We again apply this bound for $\tilde{q}>q$ instead of $q$, this time with $\tilde{q}< d+1$ and $K,\epsilon$ such that $1/\tilde{q}+2/K+\epsilon\gamma^K<1/q$. Then the claimed bound again follows by summing a geometric series. The union bound yields that this bound holds outside an exceptional set of measure at most 
\begin{align*}
\sum_{i,j}\exp(-c'M_i^2)\leq \sum_i K_i\exp(-c'M_i^2)\leq \exp(-cM^2),
\end{align*} 
due to the choice of $M_i$.
\end{proof}

\appendix
\section{Geometric series estimate}
\begin{lemma}\label{lemma Geometric series estimate}
Let $A>0$. Then we have
\begin{align*}
\sum_{k,k'\in\Z_+}\min(2^{-k-k'},A)\lesssim \begin{cases}
A(1+(\log A)^2)\quad&\mbox{if } A<1,\\
1\quad&\mbox{if } A\geq 1.
\end{cases}
\end{align*}
\end{lemma}

\begin{proof}
The case $A\geq 1$ is trivial. Assume $A<1$. We split the double sum into the obvious regions
$\Sigma_1=\{(k,k'):2^{-k-k'}\leq A\}$ and $\Sigma_2=\{(k,k'):2^{-k-k'}> A\}$. Then we have
\begin{align*}
\sum_{(k,k')\in\Sigma_1}\min(2^{-k-k'},A)=\sum_{k'\in\Z_+}2^{-k'}\sum_{k:2^{-k}\leq 2^{k'}A}2^{-k}\lesssim \sum_{k'\in\Z_+}2^{-k'}\min(1,2^{k'}A).
\end{align*}
Splitting the last sum again in the obvious way yields
\begin{align*}
\sum_{(k,k')\in\Sigma_1}\min(2^{-k-k'},A)\lesssim A(1+\log A^{-1}).
\end{align*}
Turning to the contribution of $\Sigma_2$, we have 
\begin{align*}
\sum_{(k,k')\in\Sigma_2}\min(2^{-k-k'},A)&=A\sum_{k'\in\Z_+}|\{k\in\Z_+:\,2^{-k}>2^{k'}A\}|\\
&\leq A\sum_{k'\in\Z_+}(\log A-k')_+\leq A(\log A)^2.
\end{align*}
The claim follows since $\log A^{-1}\leq 1+(\log A)^2$. 
\end{proof}

We now provide details of the calculation at the end of the proof of Lemma \ref{lemma local extension bound}. Without loss of generality we may assume that $\|V\|_{2q}=1$. By Lemma \ref{lemma Dudley}, we have

\begin{align*}
\sum_{k,k'\in \Z_+}\int \mathbf{E}\max_{\mathcal{F}_k\times \mathcal{F}_{k'}}|X_{\xi,\xi'}|\rd y\rd\tau\rd\tau'\lesssim R^{d-d/2q} \sum_{k,k'\in \Z_+}\min(2^{-k-k'},A)
\end{align*}
with $A=R^{-d+d/2q}(\log R)^{1/2}h^{d/2}$, where we recall that we are assuming that $R,h>2$. 
Since we may always assume that $R\gg 1$ and $h<R$ (otherwise there is no randomization), we have $A\ll 1$, and thus 
\begin{align*}
R^{d-d/2q} \sum_{k,k'\in \Z_+}\min(2^{-k-k'},A)\lesssim  (\log R)^{1/2}h^{d/2}(\log h+\log R)^2
\end{align*}
by Lemma \ref{lemma Geometric series estimate}.

\section{Knapp example}\label{Appendix B}
As mentioned in the introduction, we prove optimality (up to logarithms) of the key bounds of Lemmas \ref{lemma local extension bound} and \ref{lemma local resolvent bound}. In view of the foliation \eqref{foliation} it is sufficient to prove optimality of Lemma \ref{lemma local extension bound}. To this end, let $V$ be the indicator function of the tube
$$T_{R}=\{(x_1,x'):\,|x_1|<R,|x'|<R^{1/2}\},$$
normalized in $L^q$, i.e.\ $V=R^{-\frac{d+1}{2q}}\mathbf{1}_{T_{R}}$ (we will mollify this later). Here $R>1$ is a large parameter.
We consider the randomization $V_{\omega}$ (as in \eqref{randomization of deterministic V}) of this potential. We assume in the following that $\lambda=1$ in Lemma \ref{lemma local extension bound} and that $h$ is sufficiently small (to be fixed later). It is easy to see that we have 
\begin{align*}
\mathbf{E}\|\mathcal{E}^*V_{\omega}\mathcal{E}\|^2
&=\mathbf{E}\|\mathcal{E}^*\overline{V_{\omega}}\mathcal{E}\mathcal{E}^*V_{\omega}\mathcal{E}\|
= \mathbf{E}\sup_{\|f\|_{L^2(M)}=1}|\langle \mathcal{E}\mathcal{E}^*V_{\omega}\mathcal{E}f,V_{\omega}\mathcal{E}f\rangle|\\
&\geq \sup_{\|f\|_{L^2(M)}=1}|\mathbf{E}\langle \mathcal{E}\mathcal{E}^*V_{\omega}\mathcal{E}f,V_{\omega}\mathcal{E}f\rangle|\\
&\geq \sup_{\|f\|_{L^2(M)}=1}|\re\mathbf{E}\langle \mathcal{E}\mathcal{E}^*V_{\omega}\mathcal{E}f,V_{\omega}\mathcal{E}f\rangle|,
\end{align*}
where we recall that $M$ is the unit sphere in $\R^d$. In order to estimate the last expression from below,
we consider a Knapp example (see e.g.\ \cite[Example 1.8]{MR3971577})
\begin{align*}
\hat{f}_{R}(\xi):=R^{\frac{d-1}{4}}\eta(R\xi_1,R^{1/2}\xi'),
\end{align*}
where $\xi=(\xi_1,\xi')\in \R\times\R^{d-1}$ and $\eta\in C_0^{\infty}(B(0,2))$ is a nonnegative bump function equal to $1$ on $B(0,1)$. The normalization is chosen such that (up to an $R$-independent constant) $\|f\|_{L^2(M)}=1$. Assuming, as we may, that $\mathbf{E}\,\omega_i\omega_j=\delta_{ij}$, we have
\begin{align*}
\mathbf{E}\langle \mathcal{E}\mathcal{E}^*V_{\omega}\mathcal{E}f,V_{\omega}\mathcal{E}f\rangle=\sum_{j\in h\Z^d}\int_{\R^d\times \R^d}(\mathcal{E}\mathcal{E}^*)(x-y)\overline{V_{j}(y)(\mathcal{E}f)(y)}V_{j}(x)(\mathcal{E}f)(x)\rd y\rd x,
\end{align*}
where we wrote $V_j=V_{\omega}\mathbf{1}_{Q_j}$, $Q_j=j+hQ$.
Since $\mathcal{E}\mathcal{E}^*$ is proportional to convolution with $(\rd\sigma)^{\vee}$ and the latter oscillates on the unit scale, there are positive constants $r,c$ such that $\re (\rd\sigma)^{\vee}\geq c$ on $[0,r]$ (this follows from standard stationary phase asymptotics). Assume now that $2h<r$. Then, using the above Knapp example $\hat{f}_{R}$ as a test function and changing variables $u=x-y$, we obtain
\begin{align*}
\mathbf{E}\|\mathcal{E}^*V_{\omega}\mathcal{E}\|^2\gtrsim \re\sum_{j\in h\Z^d}\int_{\R^d\times \R^d} \overline{F_j(x-u)}F_j(x)\rd u\rd x,\quad (F_j=V_j\mathcal{E}f_R)
\end{align*}
up to an error involving the imaginary part $\overline{F_j(x-u)}F_j(x)$ (which is small as we will see).
At this point we consider a smooth (at the scale of $T_R$) version of the potential; this does not affect the previous arguments. What we gain by this is that now $\|\nabla F_j\|_{\infty}=\mathcal{O}(R^{-1/2})\|F_j\|_{\infty}$, whence, by Taylor expansion, 
\begin{align*}
\sum_{j\in h\Z^d}\int_{\R^d\times \R^d}\overline{F_j(x-u)}F_j(x)
=(2h)^{d}(1-\mathcal{O}(R^{-1/2}))\sum_{j\in h\Z^d}\int_{\R^d}|F_j(x)|^2\rd x.
\end{align*}
Computing the integral, this shows that
\begin{align*}
\mathbf{E}\|\mathcal{E}^*V_{\omega}\mathcal{E}\|^2\gtrsim R^{1-\frac{d+1}{q}}\|V\|_{q},
\end{align*}
which implies that $q\leq d+1$ is necessary for Lemma \ref{lemma local extension bound} to hold (since $R$ is arbitrarily large).

\bibliographystyle{abbrv}

\end{document}